\documentclass[final,onefignum,onetabnum]{siamart251216}
\usepackage{verbatim}
\usepackage{comment}
\usepackage{fancyhdr}
\usepackage{amsmath, amsfonts, amssymb}
\usepackage{mathtools}
\usepackage{bm}
\usepackage{graphicx}
\usepackage{booktabs}
\usepackage{array}
\usepackage{multirow}
\usepackage{float}
\usepackage{enumerate}
\usepackage{enumitem}
\usepackage{xcolor} % modern replacement for color
\usepackage{stmaryrd}
\usepackage{cite}
\usepackage[titletoc, title]{appendix} 
\usepackage{braket,amsfonts}
\usepackage{xspace}
\usepackage{bold-extra}
\usepackage[most]{tcolorbox}

\newsiamthm{claim}{Claim}
\newsiamremark{remark}{Remark}
\newsiamremark{hypothesis}{Hypothesis}
\crefname{hypothesis}{Hypothesis}{Hypotheses}

\Crefname{ALC@unique}{Line}{Lines}

\usepackage{amsopn}

\colorlet{texcscolor}{blue!50!black}
\colorlet{texemcolor}{red!70!black}
\colorlet{texpreamble}{red!70!black}
\colorlet{codebackground}{black!25!white!25}

\newsiamthm{lem}{Lemma}
\newsiamthm{thm}{Theorem}
\newsiamthm{cor}{Corollary}
\newsiamthm{prop}{Proposition}
\newsiamremark{rem}{Remark}
\newsiamremark{Def}{Definition}
\newsiamremark{exmp}{Example}
\newsiamthm{assum}{Assumption}
\newsiamthm{coro}{Corollary}
\numberwithin{equation}{section}

\newcommand{\ben}{\begin{align*}}
\newcommand{\een}{\end{align*}}

\newcommand{\cT}{\mathcal{T}}

\lstdefinestyle{siamlatex}{%
  style=tcblatex,
  texcsstyle=*\color{texcscolor},
  texcsstyle=[2]\color{texemcolor},
  keywordstyle=[2]\color{texemcolor},
  moretexcs={cref,Cref,maketitle,mathcal,text,headers,email,url},
}

\tcbset{%
  colframe=black!75!white!75,
  coltitle=white,
  colback=codebackground, % bottom/left side
  colbacklower=white, % top/right side
  fonttitle=\bfseries,
  arc=0pt,outer arc=0pt,
  top=1pt,bottom=1pt,left=1mm,right=1mm,middle=1mm,boxsep=1mm,
  leftrule=0.3mm,rightrule=0.3mm,toprule=0.3mm,bottomrule=0.3mm,
  listing options={style=siamlatex}
}

\newtcblisting[use counter=example]{example}[2][]{%
  title={Example~\thetcbcounter: #2},#1}

\newtcbinputlisting[use counter=example]{\examplefile}[3][]{%
  title={Example~\thetcbcounter: #2},listing file={#3},#1}

\DeclareTotalTCBox{\code}{ v O{} }
{ %fontupper=\ttfamily\color{texemcolor},
  fontupper=\ttfamily\color{black},
  nobeforeafter,
  tcbox raise base,
  colback=codebackground,colframe=white,
  top=0pt,bottom=0pt,left=0mm,right=0mm,
  leftrule=0pt,rightrule=0pt,toprule=0mm,bottomrule=0mm,
  boxsep=0.5mm,
  #2}{#1}

\patchcmd\newpage{\vfil}{}{}{}
\title{The error estimate of entropy-stable discontinuous Galerkin methods for hyperbolic conservation laws \thanks{Submitted to the editors DATE.
\funding{Y. Liu's Research is partially supported by NSFC grant 12571395, 12288201, the Strategic Priority Research Program of the Chinese Academy of Sciences under the Grant No. XDB0640000, and the Youth Innovation Promotion Association (CAS).}}}

\author{Xu-Kun Chen \thanks{Academy of Mathematics and System Sciences and School of Mathematical Science, University of Chinese Academy of Sciences, Chinese Academy of Sciences, Beijing 100190, P.R. China.  (\email{chenxukun25@mails.ucas.ac.cn}). }
\and Yong Liu \thanks{Corresponding author. ICMSEC, State Key Laboratory of Mathematical Sciences (SKLMS), Academy of Mathematics and Systems Science, and School of Mathematical Science, University of Chinese Academy of Sciences, Chinese Academy of Sciences, Beijing 100190, P.R. China. (\email{yongliu@lsec.cc.ac.cn}).}
\and Chi-Wang Shu \thanks{Division of Applied Mathematics, Brown University, Providence, RI 02912, USA. 
(\email{chi-wang\_shu@brown.edu}) }}

\headers{Error estimates of entropy stable DG methods}{X.-K. Chen, Y. Liu, and C.-W. Shu}
\ifpdf
\hypersetup{ pdftitle={The error estimate of entropy-stable discontinuous Galerkin methods for hyperbolic conservation laws} }
\fi
\allowdisplaybreaks[4]
\begin{document}
\maketitle
\begin{abstract}
Entropy inequalities are fundamental to the well-posedness of hyperbolic conservation laws, providing the essential criterion for selecting the physically admissible solution among infinitely many weak solutions. Chen and Shu [J. Comput. Phys. 345 (2017)] proposed a unified framework for constructing high-order discontinuous Galerkin (DG) methods that satisfy entropy inequalities for any given entropy via specific numerical quadrature; however, their accompanying error analysis was limited to the truncation error level, leaving a critical gap in the rigorous convergence theory for these entropy-stable schemes. This paper closes that gap by establishing rigorous a priori error estimates for semi-discrete entropy-stable DG (ESDG) methods on general unstructured meshes for hyperbolic conservation laws. The analysis applies to both scalar equations and systems, and is built upon a finite-difference-type consistency–stability argument carried out directly at the nodal level. Under a polynomial-reconstruction hypothesis and an $L^\infty$ a priori bound, we prove an $O(h^k)$  error estimate in a quadrature-based norm, which is equivalent to the broken $L^2$  norm on the finite-dimensional reconstruction space. We further extend this framework to the entropy-stable oscillation-free DG (ESOFDG) method introduced by Liu, Lu, and Shu [SIAM J. Sci. Comput. 46 (2024)], demonstrating that the additional damping terms do not degrade the convergence order. Numerical experiments suggest that the observed convergence rates may exceed the theoretical prediction by up to half an order.
\end{abstract}



\section{Introduction}
\label{sec:introduction}

Hyperbolic conservation laws govern a wide variety of physical phenomena, among which gas dynamics in the continuum regime represents a canonical and extensively studied setting. The general form of hyperbolic conservation law systems is 
\begin{subequations}\label{eq:conservation-laws}
\begin{align}
    \partial_t \mathbf{u} + \sum_{m=1}^d \partial_{x_m} \mathbf{f}_m(\mathbf u)=0,&
    \quad (\mathbf x,t)\in \Omega\times[0,T], 
    \label{eq:conservation-laws-a}
    \\
    \mathbf{u}=\mathbf{g},&
    \quad (\mathbf x,t)\in \Omega\times \lbrace 0\rbrace,
    \label{eq:conservation-laws-b}
\end{align} 
\end{subequations}
where $\Omega\subset \mathbb{R}^d, \,d \geq 1$ and \(\mathbf{u},\mathbf{f}_m\in \mathbb{R}^n, n \geq 1\) are vector-valued functions and flux functions respectively. 
A hallmark of nonlinear hyperbolic conservation laws is the spontaneous 
formation of shock waves and contact discontinuities in finite time, irrespective of the smoothness of the 
initial or boundary data. This fundamental phenomenon necessitates the framework of weak solutions, within 
which uniqueness is generically lost. Consequently, to isolate the physically relevant solution, one must 
augment the governing equations with entropy admissibility conditions. The well-posedness theory is well 
established for scalar conservation laws ($n=1$) and one-dimensional systems ($d=1$) subject to initial data 
of sufficiently small total variation, for which the entropy solution exists and is unique. By stark contrast, 
for general multi-dimensional systems, the global existence and uniqueness of entropy solutions constitute 
outstanding open problems, and the theoretical foundations remain substantially incomplete. We direct the reader 
to \cite{Dafermos2010,Godlewski1991} for an exhaustive survey. In the computational realm, although entropy 
conditions alone do not preclude non-uniqueness, the construction of discrete approximations satisfying discrete 
analogues of entropy inequalities at the grid level remains a central desideratum. This requirement defines 
the notion of entropy stability, and schemes that satisfy it are commonly termed entropy-stable (ES) schemes.

For first- and second-order discretizations, the entropy-conservative and entropy-stable 
fluxes of Tadmor \cite{1987-Tadmor-Mathcom,2003-Tadmor-ActaNum} furnish a general framework. 
In the finite volume setting, Fjordholm, Mishra, and Tadmor \cite{2012-FjordholmMishraTadmor-SIMNU} 
devised the TeCNO scheme, which combines high-order entropy-conservative fluxes 
\cite{Lefloch2002SINUM} with the sign property of essentially non-oscillatory (ENO) reconstruction 
\cite{Fjordholm2013FCM} to ensure entropy stability. More recently, attention has 
turned to entropy-stable quadrature-based discontinuous Galerkin (DG) methods. Chen 
and Shu \cite{2017-ChenShu-JCP} constructed such a method on unstructured simplex meshes 
using Gauss–Lobatto-type quadrature rules with collocated surface quadrature points and discrete 
operators endowed with the multi-dimensional summation-by-parts (SBP) property 
\cite{Fernandez2018JSC,Hicken2016SISC}. Several related entropy-stable DG methods 
within the SBP framework have subsequently been proposed 
\cite{Crean2017AIAA,Crean2018JCP,Chan2018JCP,Chan2019SISC}; the reader 
is referred to \cite{2020-ChenShu-CSIAM} for a comprehensive survey. 
The paradigm has further been extended to encompass the Navier-Stokes 
equations \cite{Carpenter2014SISC,Gassner2018JSC}, magnetohydrodynamics (MHD) 
\cite{Bohm2020JCP, Liu2018JCP}, shallow water equations \cite{Gassner2016AMC,Wintermeyer2017JCP}, 
Euler equations with gravity \cite{2025-LiuGuoJiangZhang-JCP,2026-LiuGuoJiangZhang-JSC}, gradient 
flows \cite{Sun2019KRM}, two-phase flows \cite{Renac2019JCP}, and stochastic problems \cite{Offner2018M2AN}.

Although entropy stability of the ESDG method is well established, rigorous error estimates for smooth solutions remain unavailable, to the best of our knowledge. For semi-discrete DG and Runge–Kutta DG schemes, such estimates have been extensively developed; see, for example, \cite{2004-ZhangShu-SINUM,2006-ZhangShu-SINUM,2010-ZhangShu-SINUM}. The effect of inexact quadrature on DG error analysis has also been addressed in \cite{2017-HuangShu-NMPDE}; however, the results therein do not extend to general entropy-conservative SBP flux-differencing schemes. Recently, Worku, Del Rey Fernández, and Zingg \cite{2026-Worku-arxiv} established convergence of entropy-conservative split-form discretizations of symmetric hyperbolic systems with homogeneous flux functions that possess globally bounded second derivatives, within the multi-dimensional continuous-SBP framework. Ranocha \cite{2026-Ranocha-arxiv} subsequently extended this analysis to diagonal-norm SBP operators on curved meshes and derived an a priori $h^p$ error estimate for entropy-conservative high-order SBP semi-discretizations of nonlinear symmetrizable systems via a discrete relative-entropy argument. Neither of these analyses, however, encompasses the  Chen-Shu ESDG method  \cite{2017-ChenShu-JCP} with general entropy-stable interface fluxes or the 
oscillation-free extension \cite{2024-LiuLuShu-SISC} analyzed below. The fundamental obstacle is that classical DG error analysis relies crucially on the assumption that the numerical solution resides in a prescribed space of piecewise polynomials of degree $k$. For the entropy-stable nodal DG method introduced in \cite{2017-ChenShu-JCP}, this assumption is violated: in multiple space dimensions, the nodal values produced by the scheme do not, in general, correspond to a piecewise polynomial of degree $k$. Consequently, existing error estimates for DG methods with inexact quadrature are neither directly applicable nor adequate to establish high-order convergence for the ESDG method. Nevertheless, the nodal flux-differencing formulation and SBP operators impart a finite-difference-type structure to the ESDG method, thereby motivating a consistency-stability analysis formulated directly at the nodal level.

In this paper, we prove an a priori \(h^k\) error estimate in a quadrature-based norm equivalent to the broken \(L^2\)  norm on the reconstruction space for smooth solutions of the semi-discrete ESDG method applied to \eqref{eq:conservation-laws}. The result covers both scalar conservation laws and systems under a polynomial reconstruction hypothesis and an \(L^\infty\)  a priori bound, and extends to the entropy-stable oscillation-free DG (ESOFDG) method \cite{2024-LiuLuShu-SISC}, where the additional damping terms are shown not to degrade the convergence rate. The proof follows a consistency–stability approach: we first derive a local truncation error by inserting the nodal values of the exact solution into the semi-discrete scheme, then formulate the error equation and bound the volume, interface, and truncation contributions separately. The leading volume terms are controlled through the skew-symmetric SBP structure combined with entropy symmetrization of the flux Jacobians; the remaining nonlinear terms are estimated via Taylor's expansions and mesh-scaled matrix-norm bounds; and the interface contributions are managed using the entropy stability of the numerical flux. The resulting differential inequality is closed by a Gronwall argument. We further justify the polynomial reconstruction hypothesis and a priori assumption by continuation arguments. For the ESOFDG method, estimates of the damping coefficients together with the projection structure of the damping terms ensure that the additional dissipation does not affect the convergence order. Numerical experiments, however, indicate that the observed convergence rates may be 
higher than the theoretical bound proved by up to half an order. 
This suggests that the present estimate, 
which is indeed derived from the truncation error, may not be sharp. 
Similar phenomena have been reported in the error analysis of DG-type methods, 
indicating the possible presence of more delicate cancellation 
or superconvergence mechanisms \cite{2003-ZhangShu-M3AS}.

This paper is organized as follows. In Section~\ref{sec:preliminaries}, we introduce some basic notations,
the semi-discrete ESDG method, and some useful lemmas. In Section~\ref{sec:scalar-case}, we first 
derive the error estimate for the ESDG method for scalar conservation laws, then give the detailed 
discussion on our assumptions and extend our proof to the so-called ESOFDG scheme.
In Section~\ref{sec:system-case}, we prove the system counterpart of Section~\ref{sec:scalar-case}. 
Numerical results are reported in Section~\ref{sec:numerical}. We finally give concluding remarks
in Section~\ref{sec:conclude}. Some technical proofs of lemmas are left in Appendix~\ref{sec:app}.

\section{Preliminaries}
\label{sec:preliminaries}

In this section, we will first give a brief review of the entropy analysis for \eqref{eq:conservation-laws} 
on the PDE level. Then we introduce the quadrature rules and the SBP operators \cite{2020-ChenShu-CSIAM}, 
which mimic
integration by parts at the discrete level. After that, we can describe the ESDG method. At the end of this 
section, we will give some useful auxiliary results that will be used.

\subsection{Entropy analysis}

A strictly convex function \(U(\mathbf{u})\) is called an entropy function for \eqref{eq:conservation-laws} 
if there exist entropy fluxes, \( \mathbf{F}(\mathbf{u})=( F_1(\mathbf{u}), F_2(\mathbf{u}),\ldots, F_d(\mathbf{u}))^T\), such that
\begin{equation}
    \label{eq:entropy-pair}
    U^\prime(\mathbf{u}) \mathbf{f}_m'(\mathbf{u})=F_m'(\mathbf{u}),
    \quad 
    m=1,\ldots,d.
\end{equation}
Here \(U'(\mathbf{u})\) and \(F_m'(\mathbf{u})\) are understood as row vectors, 
and \(\mathbf{f}_m'(\mathbf{u})\) is the \(n\times n\) Jacobian matrix.
We call the \((U(\mathbf{u}),\mathbf{F}(\mathbf{u}))\) an entropy pair.
The entropy variable and the entropy Hessian are defined by
\[
    \mathbf v(\mathbf u):=U'(\mathbf u)^T,
    \quad
    H(\mathbf u):=U''(\mathbf u).
\]
In the error analysis below, we assume uniform convexity on
the relevant set of states: there exist constants \(C_1,C_2>0\) such
that
\begin{equation*}
    C_1 I_n \leq H(\mathbf u) \leq C_2 I_n,
\end{equation*}
where \(I_n\) is the \(n\times n\) identity matrix. Consequently,
\(H(\mathbf u)\) is symmetric positive-definite, and the mapping
\(\mathbf u\mapsto\mathbf v(\mathbf u)\) is invertible.
We also define the entropy potential fluxes as follows:
\begin{equation*}
    \psi_m(\mathbf{v}):=\mathbf{v}^T \mathbf{f}_m (\mathbf{u}(\mathbf{v}))-F_m(\mathbf{u}(\mathbf{v})),
    \quad m=1,\ldots,d.
\end{equation*}
One can easily verify that
$\psi_m'(\mathbf{v})=\mathbf{f}_m(\mathbf{u}(\mathbf{v}))^T$.
In addition, for a unit vector \(\mathbf n=(n_1,\ldots,n_d)^T\), we set
\[
    \mathbf{f}_{\mathbf n}(\mathbf{u}):=\sum_{m=1}^d n_m \mathbf{f}_m(\mathbf{u}),
    \quad
    F_{\mathbf n}(\mathbf{u}):=\sum_{m=1}^d n_m F_m(\mathbf{u}),
    \quad
    \psi_{\mathbf n}(\mathbf{v}):=\sum_{m=1}^d n_m \psi_m(\mathbf{v}).
\]
In general, the physical relevant solution \(\mathbf{u}\) of \eqref{eq:conservation-laws} 
is meant to satisfy the following entropy inequality
\begin{equation}
    \partial_t U(\mathbf{u})+\sum_{m=1}^{d}\partial_{x_m} F_m(\mathbf{u}) \leq 0,
    \label{eq:entropy-inequality}
\end{equation} 
in the distribution sense for all entropy pairs.
Formally integrating \eqref{eq:entropy-inequality} in space with the assumption that \(\mathbf{u}\) is 
compactly supported, we obtain
\begin{equation*}
    \frac{d}{dt}\int_{\Omega} U(\mathbf{u}) dx \leq 0,
\end{equation*}
which indicates that the total amount of entropy is non-increasing in time.
For scalar conservation laws, any convex function \(U\) defines an entropy function with entropy fluxes
\(F_m(u)=\int^u U'(s)f_m'(s)ds\). For hyperbolic systems, the existence of entropy functions is equivalent to the symmetrizability of the systems 
\cite{2013-GRAMS-springer}. Fortunately, for most systems we are interested in, such as shallow water equations, 
compressible Euler equations, and MHD equations, we are able to find the entropy functions with
physical meaning. 

For numerical aspects, we would like to find our semi-discrete scheme
to satisfy \eqref{eq:entropy-inequality}. Such property of schemes is referred to as
\textit{entropy-stable} (ES). To design ES schemes, the following entropy conservative fluxes and entropy stable fluxes 
are proposed by Tadmor \cite{1987-Tadmor-Mathcom, 2003-Tadmor-ActaNum}, and defined as follows.

\begin{Def}
  For \(m=1,\ldots,d\), a two-point numerical flux \(\mathbf{f}_{m,S}(\mathbf{u}_L,\mathbf{u}_R)\) 
is called entropy conservative with respect to \(U\) if it is consistent, symmetric, and satisfies
\begin{equation*}
    (\mathbf{v}_R-\mathbf{v}_L)^T \mathbf{f}_{m,S}(\mathbf{u}_L,\mathbf{u}_R)=\psi_{m,R}-\psi_{m,L},
\end{equation*}
where $\mathbf{v}_{L,R}$ and $\psi_{m,L,R}$ are the entropy variables and entropy potential fluxes at left and right states. Furthermore, given entropy conservative fluxes in all coordinate directions, we define the directional entropy 
conservative flux by
\begin{align*}
    \mathbf{f}_{\mathbf n,S}(\mathbf u_L,\mathbf u_R):=\sum_{m=1}^d n_m \mathbf{f}_{m,S}(\mathbf u_L,\mathbf u_R).
\end{align*}
\end{Def}
\begin{Def}
    A consistent two-point numerical flux \(\hat {\mathbf f}_{\mathbf n}(\mathbf{u}_L,\mathbf{u}_R)\) is called entropy stable 
with respect to \(U\) if it satisfies
\begin{equation*}
    (\mathbf{v}_R-\mathbf{v}_L)^T\hat {\mathbf f}_{\mathbf n}(\mathbf{u}_L,\mathbf{u}_R)
    \le \psi_{\mathbf n,R}-\psi_{\mathbf n,L}.
\end{equation*}
\end{Def}

\subsection{Quadrature rules and SBP operators}
\label{subsec: SBP-intro}

In this subsection, we will briefly review the quadrature rules and SBP operators in order to rewrite
the DG method under the SBP framework. Assume that \(\Omega\subset\mathbb R^d\) is a bounded polygonal
domain. Let \(\{\mathcal T_h\}_{h>0}\) be a family of conforming,
shape-regular, and quasi-uniform simplicial partitions of \(\Omega\).
We denote $h_K:=\operatorname{diam}(K),\, h:=\max_{K\in\mathcal T_h} h_K.$ Let \(\mathcal E_h\) be the set of all interior faces of \(\mathcal T_h\), and \(\mathcal E_K\) be the set of all faces of \(K\in \mathcal{T}_h\).
For each \(\gamma\in\mathcal E_h\), we fix a unit normal vector \(\mathbf n_\gamma\).
If \(\gamma=K^-\cap K^+\), then labels \(K^-\) and \(K^+\) are chosen such that  
\(\mathbf n_\gamma\) points from \(K^-\) to \(K^+\). 
For a piecewise smooth function \(w\), we denote the two traces of \(w\) on \(\gamma\) by
\[
    w^-:=w|_{K^-},\quad w^+:=w|_{K^+}.
\]
When a single element \(K\) is fixed and \(\gamma\in\mathcal E_K\), we also write
$w^{\gamma,K}:=w|_K$ on $\gamma$, 
namely the trace of \(w\) on \(\gamma\) taken from the inner side of \(K\). 
Let \(\widehat K\) be a fixed reference simplex and, for each \(K\in\mathcal T_h\), let
\begin{equation}
    \label{eq:affine-map}
    F_K(\widehat{\mathbf x})
    =
    A_K\widehat{\mathbf x}+\mathbf b_K
    =: \mathbf{x}
\end{equation}
be the affine mapping from \(\widehat K\) onto \(K\). We assume that
the volume and surface quadrature rules, as well as the corresponding
nodal SBP operators on \(K\), are obtained by affine transformation
of fixed reference-element rules and operators. For notational
convenience, all definitions below are stated directly on the element \(K\).
For each simplex \(K\in \mathcal{T}_h\), suppose that there is a quadrature rule of degree at least \(2k-1\)
on \(K\), associated with quadrature nodes and positive weights \(\{(\mathbf{x}^K_j,\omega^K_j)\}_{j=1}^{N_{Q,k}}\),
\(N_{Q,k}\geq N_{P,k}\) with
\[
    N_{P,\nu}= \operatorname{dim} \mathcal{P}^\nu(\mathbb{R}^d),\quad 0\leq \nu\leq k.
\]
For each \(\gamma\in \mathcal{E}_h\), suppose that there is a quadrature rule of degree at least \(2k\) on
\(\gamma\), associated with quadrature nodes and positive weights \(\{(\mathbf{x}^\gamma_j,\tau^\gamma_j)\}_{j=1}^{N_{B,k}}\).
Now we take \(\{ p_l \}_{l=1}^{N_{P,\nu}}\) as the set of basis functions of \(\mathcal{P}^\nu(\mathbb{R}^d)\),
which means
\[
    \{ p_l \}_{l=1}^{N_{P,0}} \subset \{ p_l \}_{l=1}^{N_{P,1}} \subset \cdots \subset \{ p_l \}_{l=1}^{N_{P,k}}.
\]
Then we define the Vandermonde matrices, whose columns are nodal values:
\[
    V^K_{\nu}=[\vec{p}_1^K,\cdots,\vec{p}_{N_{P,\nu}}^K],
    \quad
    V^\gamma_{\nu}=[\vec{p}_1^\gamma,\cdots,\vec{p}_{N_{P,\nu}}^\gamma],
\]
where \(\vec{p}_l^K=[p_l(\mathbf{x}^K_1),\cdots,p_l(\mathbf{x}^K_{N_{Q,k}})]^T, 
\vec{p}_l^\gamma=[p_l(\mathbf{x}^\gamma_1),\cdots,p_l(\mathbf{x}^\gamma_{N_{B,k}})]^T\).
We also define polynomial differential matrices \(\hat{D}_m\) such that
\[
    \frac{\partial}{\partial x_m}p_l (\mathbf{x}) = \sum_{r=1}^{N_{P,k}}\hat{D}_{m,rl} p_r(\mathbf{x}).
\]
Then \(V_k^K \hat{D}_m\) is the Vandermonde matrix of \(\{ \partial_{x_m} p_l (\mathbf{x}) \}_{l=1}^{N_{P,k}} \)
on \(K\). 
For a scalar function \(w\), define the volume and surface nodal vectors by
\[
    \vec w^K
    :=
    \bigl[w(\mathbf x_1^K),\ldots,w(\mathbf x_{N_{Q,k}}^K)\bigr]^T,
    \quad
    \vec w^{\gamma,K}
    :=
    \bigl[w(\mathbf x_1^\gamma),\ldots,w(\mathbf x_{N_{B,k}}^\gamma)\bigr]^T.
\]
Denote
\[
    M^K:=\operatorname{diag}(\omega_1^K,\ldots,\omega_{N_{Q,k}}^K),
    \quad
    B^\gamma:=\operatorname{diag}(\tau_1^\gamma,\ldots,\tau_{N_{B,k}}^\gamma).
\]
We then define the continuous and discrete inner products on \(K\) and \(\gamma\) that
\begin{equation*}
    (u,w)_K:=\int_K uw dx,
    \quad
    (u,w)_{K,\omega}
    :=\sum_{j=1}^{N_{Q,k}}\omega_j^K u(\mathbf x_j^K)w(\mathbf x_j^K)
    =(\vec u^K)^T M^K\vec w^K,
\end{equation*}
\begin{equation*}
    (u,w)_\gamma:=\int_\gamma uw ds,
    \quad
    \langle u,w\rangle_{\gamma,\tau}
    :=\sum_{s=1}^{N_{B,k}}\tau_s^\gamma u(\mathbf x_s^\gamma)w(\mathbf x_s^\gamma)
    =(\vec u^{\gamma,K})^T B^\gamma \vec w^{\gamma,K}.
\end{equation*}

In order to obtain the \textit{nodal} SBP property, we recall the definitions of \textit{degree k difference matrix}
\(D_m^K\) and extrapolation matrices \(\{ R^{\gamma,K}\}_{\gamma\in \mathcal{E}_K}\), for which the following two conditions
hold:
\begin{enumerate}
    \item Exactness: both \(D_m^K\) and \(R^{\gamma,K}\) are exact for polynomials of degree \( k\), i.e.,
        \begin{equation}
            D_m^K V_k^K=V_k^K\hat{D}_m,\quad R^{\gamma,K} V_k^K= V_k^\gamma.
            \label{eq: exact-for-nodal-SBP}
        \end{equation}
    \item Summation-by-parts: setting \(S_m^K:=M^K D_m^K\) and \(E^{\gamma,K}:=(R^{\gamma,K})^T B^\gamma R^{\gamma,K}\), we have
        \begin{equation}
            S_m^K+(S_m^K)^T=\sum_{\gamma\in\mathcal E_K} n_m^{\gamma,K} E^{\gamma,K}.
            \label{eq: nodal-SBP}
        \end{equation}
\end{enumerate}
We also define \(L^2\) projection matrices under the discrete
inner product: 
\begin{equation*}
    P^K_\nu = \big( (V_\nu^K)^T M^K V_\nu^K\big)^{-1}(V_\nu^K)^T M^K,\quad 0\leq \nu\leq k.
\end{equation*}
In particular, for \(\nu=0\), we have the orthogonality of the operator \(P_0^K\) that
\begin{equation*}
    \vec{c}^T M^K (\vec{w}^K-V_0^KP_0^K \vec{w}^K)=0, \quad \forall \vec{w}^K\in \mathbb{R}^{N_{Q,k}},
\end{equation*}
where \(\vec{c}\in \mathbb{R}^{N_{Q,k}}\) is a constant vector. 
The existence of SBP difference matrices is ensured by the following theorem \cite{2020-ChenShu-CSIAM}.
\begin{thm}
    Assume that we have extrapolation matrices \(R^{\gamma,K}\) satisfying the exactness property. Then the difference 
    matrices, given by the formula
    \begin{equation*}
    D_m^K
    =
    \frac12 (M^K)^{-1}
    \sum_{\gamma \in \mathcal E_K}
    n_m^{\gamma,K}
    \bigl(R^{\gamma,K}+V_k^\gamma P_k^K\bigr)^T
    B^\gamma
    \bigl(R^{\gamma,K}-V_k^\gamma P_k^K\bigr)
    +
    V_k^K\widehat D_m P_k^K.
\end{equation*}
satisfy \eqref{eq: exact-for-nodal-SBP} and \eqref{eq: nodal-SBP}.
\end{thm}
For the choice of the extrapolation matrices \(R^{\gamma,K}\), the readers can refer to \cite{2020-ChenShu-CSIAM}. 
According to the construction of \(P_0^K\) and \(V_0^K\), 
one can easily check that \(Q^K:=I_{N_{Q,k}}-V_0^KP_0^K\) satisfy
\begin{equation}
    \label{eq:QK-ortho}
    Q^K=(Q^K)^2, \quad (M^K Q^K)^T=M^KQ^K. 
\end{equation}
For the vector-valued function \(\mathbf{w}:\mathbb{R}^d \rightarrow \mathbb{R}^n\), we introduce the vector of nodal values
\begin{align*}
        &\vec{\mathbf w}^K
        :=
        \begin{bmatrix}
            \mathbf w(\mathbf x_1^K)^T,\ldots,
            \mathbf w(\mathbf x_{N_{Q,k}}^K)^T
        \end{bmatrix}^T
        \in \mathbb{R}^{n N_{Q,k}},
        \\
        &\vec{\mathbf w}^{\gamma,K}
        :=
        \begin{bmatrix}
            \mathbf w^{\gamma,K}(\mathbf x_1^\gamma)^T,\ldots,
            \mathbf w^{\gamma,K}(\mathbf x_{N_{B,k}}^\gamma)^T
        \end{bmatrix}^T
        \in \mathbb{R}^{n N_{B,k}}.
\end{align*}
as well as the Kronecker products
\begin{equation*}
    \begin{aligned}
        &\mathbf{M}^K=M^K \otimes I_n, \quad \mathbf{B}^\gamma=B^\gamma \otimes I_n, \quad
        \mathbf{D}_m^K=D^K_m \otimes I_n, \quad \mathbf{R}^{\gamma,K}=R^{\gamma,K} \otimes I_n, \\
        &\hat{\mathbf{M}}^K=\hat{M}^K \otimes I_n, \quad \hat{\mathbf{D}}_m=\hat{D}_m \otimes I_n, \quad
        \mathbf{V}_\nu^K=V_\nu^K \otimes I_n, \quad \mathbf{V}^{\gamma}_\nu=V^{\gamma}_\nu \otimes I_n,\\
        & \mathbf{Q}^K= Q^K\otimes I_n, \quad
        \mathbf P_\nu^K:=P_\nu^K\otimes I_n,
        \quad
        \mathbf S_m^K:=S_m^K\otimes I_n.
    \end{aligned}
\end{equation*}
All the properties before remain valid for vector-valued nodal data
under the Kronecker-product convention introduced above.

For a volume nodal vector \(\vec{\mathbf w}^K\), we define the discrete \(h\)-norm by
\[
    \|\vec{\mathbf w}^K\|_{h,K}^2
    :=
    (\vec{\mathbf w}^K)^T
    \mathbf M^K\vec{\mathbf w}^K,
    \quad
    \|\vec{\mathbf w}\|_h^2
    :=
    \sum_{K\in\mathcal T_h}\|\vec{\mathbf w}^K\|_{h,K}^2,
\]
where the corresponding global broken nodal vector is
\[
    \vec{\mathbf w}
    :=
    \bigl(\vec{\mathbf w}^{\,K}\bigr)_{K\in\mathcal T_h}
    \in
    \prod_{K\in\mathcal T_h}\mathbb R^{nN_{Q,k}}.
\]  
We also define the \(\ell^q,1\leq q\leq \infty\) norms by
\[
    \|\vec{\mathbf w}^K\|_{\ell^q,K}^q
    :=
    \sum_{j=1}^{N_{Q,k}}|\mathbf w_j^K|^q,
    \quad 1\le q<\infty,
\]
and
\[
    \|\vec{\mathbf w}^K\|_{\ell^\infty,K}
    :=
    \max_{1\le j\le N_{Q,k}}|\mathbf w_j^K|,
\]
where \(\mathbf{w}_j^K=\mathbf{w}(\mathbf{x}_j^K)\) and \(|\cdot|\) denotes the Euclidean norm.
The corresponding global norms are
\[
    \|\vec{\mathbf w}\|_{\ell^q}^q
    :=
    \sum_{K\in\mathcal T_h}\|\vec{\mathbf w}^K\|_{\ell^q,K}^q,
    \quad
    \|\vec{\mathbf w}\|_{\ell^\infty}
    :=
    \max_{K\in\mathcal T_h}\|\vec{\mathbf w}^K\|_{\ell^\infty,K}.
\]
\begin{rem}[Affine transformation of the SBP operators]
\label{rem:affine-transformation-sbp}
Under the affine mapping \eqref{eq:affine-map}, and
let \(\widehat\gamma \in \mathcal{E}_{\widehat K}\) be the reference face
mapped onto \(\gamma \in \mathcal{E}_K\), and set
\[
    J_K:=|\det A_K|,
    \quad
    J_{\gamma,K}:=J_K\bigl|(A_K^{-1})^{T}\widehat{\mathbf n}_{\widehat\gamma}\bigr|.
\]
After a consistent ordering of the volume and face nodes, we obtain
\begin{equation}
    \label{eq:affine-map-sbp}
    \begin{aligned}
    &M^K=J_K M^{\widehat K},
    \quad
    B^\gamma=J_{\gamma,K}B^{\widehat\gamma},
    \\
    &D_m^K=\sum_{\alpha=1}^d ((A_K^{-1})^{T})_{m\alpha}D_\alpha^{\widehat K},
    \quad
    R^{\gamma,K}=R^{\widehat\gamma,\widehat K},
    \quad
    Q^K=Q^{\widehat K}.
    \end{aligned}
\end{equation}
\end{rem}

\subsection{Entropy stable DG schemes with collocated surface nodes}
\label{subsec:collocated-esdg}

We now introduce the ESDG scheme that is the focus of this paper. We first recall the classical DG method. Given polynomial degree \(k\geq 0\), we define the DG finite element space,
\begin{equation*}
    \mathbf{W}_h^k=\{\mathbf{w}_h: \mathbf{w}_h|_K=\mathbf{w}_h^K\in \Big[ \mathcal{P}^k(K)\Big]^n,K\in \mathcal{T}_h \}.
\end{equation*}
For \eqref{eq:conservation-laws}, the classical DG method can be constructed as follows:
we seek \(\mathbf u_h \in \mathbf{W}_h^k\) such that
for each \(\mathbf w_h \in \mathbf{W}_h^k\) and
\(K\in \mathcal{T}_h\), we have
\begin{equation}
\label{eq:classic-DG}
    (\partial_t \mathbf{u}_h^K, \mathbf{w}_h^K)_K-
    \sum_{m=1}^d (\mathbf{f}_m(\mathbf{u}_h^K),\partial_m \mathbf{w}_h^K)_K
    =
    -\sum_{\gamma\in\mathcal E_K}
    \langle \hat{\mathbf{f}}_{\mathbf{n}}(\mathbf{u}_h^{\gamma,K},\mathbf{u}_h^{\gamma,\tilde{K}}), \mathbf{w}_h^{\gamma,K}\rangle_\gamma
    ,
\end{equation}
where \(\tilde{K}\) denotes the neighboring element sharing the face \(\gamma\) with \(K\), and \(\hat{\mathbf{f}}_{\mathbf{n}}\) is the interface numerical flux function. By applying the volume quadrature rule and the surface quadrature rule to \eqref{eq:classic-DG} and replacing the standard difference operator with a flux-differencing term, we obtain the ESDG method introduced in \cite{2017-ChenShu-JCP},
\begin{equation}
\label{eq:esdg-scheme}
    \mathbf{M}^K\frac{d\vec{\mathbf u}_h^K}{dt}
    +
    2\sum_{m=1}^d
    \bigl(
       \mathbf{S} _m^K\circ \mathbf{F}_{m,S}(\vec{\mathbf u}_h^K,\vec{\mathbf u}_h^K)
    \bigr)\vec{\mathbf 1}^K
    =
    \sum_{\gamma\in\mathcal E_K}
    ( \mathbf{R}^{\gamma,K})^T \mathbf{B} ^\gamma
    \left(
        \vec{\mathbf f}_{\mathbf n}^{\gamma,K}
        -
        \vec{\hat{\mathbf f}}_{\mathbf n}^{\gamma,K}
    \right).
\end{equation}
Here \(\circ\) denotes the Hadamard product and \(\vec{\mathbf 1}^K\) is the vector of ones,
and \(\mathbf{F}_{m,S}(\cdot,\cdot)\)  is the matrix of pairwise combinations of entropy conservative fluxes
\begin{equation*}
    \mathbf{F}_{m,S}(\vec{\mathbf{u}}_L,\vec{\mathbf{u}}_R)=
    \begin{bmatrix}
        \operatorname{diag}(\mathbf{f}_{m,S}(\mathbf{u}_{L,1},\mathbf{u}_{R,1})) & \cdots & \operatorname{diag}(\mathbf{f}_{m,S}(\mathbf{u}_{L,1},\mathbf{u}_{R,\mathcal{N}_R})) \\
        \vdots & \ddots & \vdots \\
        \operatorname{diag}(\mathbf{f}_{m,S}(\mathbf{u}_{L,\mathcal{N}_L},\mathbf{u}_{R,1})) & \cdots & \operatorname{diag}(\mathbf{f}_{m,S}(\mathbf{u}_{L,\mathcal{N}_L},\mathbf{u}_{R,\mathcal{N}_R})) 
    \end{bmatrix},
\end{equation*}
for \(\vec{\mathbf{u}}_L\in \mathbb{R}^{n \mathcal{N}_L}\) and \(\vec{\mathbf{u}}_R\in \mathbb{R}^{n \mathcal{N}_R}\). The $\vec{\mathbf f}_{\mathbf n}^{\gamma,K}$ and $
        \vec{\hat{\mathbf f}}_{\mathbf n}^{\gamma,K}$ are defined as 
\begin{align*}
    &\vec{\mathbf f}_{\mathbf n}^{\gamma,K}=
    \begin{bmatrix}
        (\mathbf f_{\mathbf n}(\mathbf{u}_{h,1}^{\gamma,K}))^T,\ldots,
        (\mathbf f_{\mathbf n}(\mathbf{u}_{h,N_{B,k}}^{\gamma,K}))^T
    \end{bmatrix}^T,\\
    &\vec{\hat{\mathbf{f}}}_{\mathbf{n}}^{\gamma,K}=
    \begin{bmatrix}
        (\hat{\mathbf{f}}_{\mathbf{n}}(\mathbf{u}_{h,1}^{\gamma,K},\mathbf{u}_{h,1}^{\gamma,\tilde{K}}))^T,\ldots,
        (\hat{\mathbf{f}}_{\mathbf{n}}(\mathbf{u}_{h,N_{B,k}}^{\gamma,K},\mathbf{u}_{h,N_{B,k}}^{\gamma,\tilde{K}}))^T
    \end{bmatrix}^T.
\end{align*}
Note that this method requires the collocated surface quadrature nodes, namely, 
\(\{\mathbf{x}^\gamma_j\}_{j=1}^{N_{B,k}}\subset\{\mathbf{x}^K_j\}_{j=1}^{N_{Q,k}}\).
Thus \(\mathbf R^{\gamma,K}\) is simply a restriction onto \(\gamma\).
Under the assumptions that \(\mathbf f_{m,S}\) is entropy
conservative and that
\(\widehat{\mathbf f}_{\mathbf n}\) is entropy stable with respect
to \(U\), the
\eqref{eq:esdg-scheme} satisfies a semi-discrete entropy inequality \cite{2017-ChenShu-JCP}. Next, 
we introduce some auxiliary results that will be used in 
the error estimates of \eqref{eq:esdg-scheme}.
Here and below, for two nonnegative quantities \(a\) and
\(b\), we write \(a\lesssim b\) if there exists a constant \(C>0\),
independent of \(h\), such that
\(
    a\le Cb.
\)
Moreover, we write \(a\sim b\) if both \(a\lesssim b\) and \(b\lesssim a\) are true.
We first give the scaling properties of the quadrature weights and SBP operators.
\begin{lem}[Scaling properties of SBP operators]
\label{lem:scaling-sbp}
    Assume that the mesh is shape regular and quasi-uniform; the following estimates
    hold uniformly for all \(K\in\mathcal T_h\) and all faces \(\gamma\in \mathcal{E}_K\):
    \[
        \omega_j^K \sim h_K^d,\quad
        \tau_s^\gamma\sim h_K^{d-1}.
    \]
    Consequently, for any volume nodal vector \(\vec{\mathbf w}^K\),
    \[
        \|\vec{\mathbf w}^K\|_{h,K}^2
        \sim h_K^d\|\vec{\mathbf w}^K\|_{\ell^2,K}^2.
    \]
    For SBP operators, for any \(1\le q\le\infty\),
    \[
        \|\mathbf D_m^K\|_{\ell^q}\lesssim h_K^{-1},
        \quad
        \|\mathbf M^K\|_{\ell^q}\lesssim h_K^{d},
        \quad
        \|\mathbf R^{\gamma,K}\|_{\ell^q}\lesssim 1,
        \quad
        \|\mathbf Q^K\|_{\ell^q} \lesssim 1.
    \]
    Here the operator norms are induced by the corresponding nodal
    \(\ell^q\)-norms. In addition, since \(\mathbf Q^K\) is an
    \(\mathbf M^K\)-orthogonal projection,
    \[
        \|\mathbf Q^K\vec{\mathbf w}^{\,K}\|_{h,K}
        \leq
        \|\vec{\mathbf w}^{\,K}\|_{h,K}.
    \]
\end{lem}
The proof of Lemma~\ref{lem:scaling-sbp} is given in the appendix; 
see Section~\ref{app:lem-scaling}.
We next recall the following inverse estimates \cite{1978-PG-FEM}
for piecewise polynomials.
\begin{lem}[Inverse inequalities]\label{lem:inverse-inequalities}
There exists a constant \(C>0\), independent of \(h\), such that for any
\(\mathbf{v}_h \in \mathbf{W}_h^k\),
\[
    |\mathbf{v}_h|_{H^1} \lesssim h^{-1}\|\mathbf{v}_h\|_{L^2},\quad
    \|\mathbf{v}_h\|_{L^\infty} \lesssim h^{-d/2}\|\mathbf{v}_h\|_{L^2}.
\]
\end{lem}
The next lemma gives the local truncation estimate of the ESDG method, see \cite{2017-ChenShu-JCP}.

\begin{lem}[Local truncation errors of the ESDG method]\label{lem:local-truncation-error}
Assume that the exact solution \(\mathbf u\), physical fluxes \(\mathbf{f}_m\) and the entropy conservative fluxes \(\mathbf f_{m,S}\) are sufficiently smooth, and \(\hat{\mathbf{f}}_m\) are locally Lipschitz continuous with respect to both arguments. Let
\begin{align*}
    \boldsymbol\tau_j^K(\mathbf u)
    :=&  \frac{d \mathbf u_j^K}{dt}
    +2\sum_{m=1}^d\sum_{\ell=1}^{N_{Q,k}}
    D_{m,j\ell}^K \mathbf f_{m,S}(\mathbf u_j^K,\mathbf u_\ell^K) \\
    &-\sum_{\gamma\in\mathcal E_K}\sum_{s=1}^{N_{B,k}}
    R_{sj}^{\gamma,K}\frac{\tau_s^\gamma}{\omega_j^K}
    \bigl(
        \mathbf f_{\mathbf n}(\mathbf u_s^{\gamma,K})
        -\hat{\mathbf f}_{\mathbf n}
        (\mathbf u_s^{\gamma,K},\mathbf u_s^{\gamma,\tilde{K}})
    \bigr),
\end{align*}
for \(j=1,\ldots,N_{Q,k}.\) Then we have
\begin{equation*}
    \|\vec{\boldsymbol\tau}^K(\mathbf u)\|_{\ell^q,K}\lesssim h^k,
    \quad 1\le q\le\infty,
\end{equation*}
where $\vec{\boldsymbol\tau}^K(\mathbf u)=
\begin{bmatrix}
    \boldsymbol\tau_1^K(\mathbf u)^T,\ldots,
    \boldsymbol\tau_{N_{Q,k}}^K(\mathbf u)^T
\end{bmatrix}^T$.
\end{lem}
The next lemma, intuitively, 
measures the difference between the central flux and the entropy conservative flux.

\begin{lem}\label{lem:Phi-structure}
Define
\[
    \mathbf f_{m,c}(\mathbf a,\mathbf b)
    :=
    \frac{\mathbf f_m(\mathbf a)+\mathbf f_m(\mathbf b)}{2},
    \quad
    \boldsymbol\Phi_m(\mathbf a,\mathbf b)
    :=
    \mathbf f_{m,S}(\mathbf a,\mathbf b)-\mathbf f_{m,c}(\mathbf a,\mathbf b),
    \quad \mathbf a,\mathbf b\in\mathbb R^n.
\]
Assume that \(\mathbf{f}_{m}\in C^3(\mathbb{R}^n)\), 
\(\mathbf{f}_{m,S}\in C^3(\mathbb{R}^n \times \mathbb{R}^n)\) 
and \(\mathbf{f}_{m,S}\) is the symmetric and consistent flux. 
For any convex compact set \( S \subset \mathbb R^n\), 
if \(\mathbf a,\mathbf b,\mathbf c,\mathbf d\in S\) and satisfy
\[
    |\mathbf a-\mathbf b|+|\mathbf c-\mathbf d|\lesssim h,
\]
then
\begin{equation*}
    |\boldsymbol\Phi_m(\mathbf a,\mathbf b)-\boldsymbol\Phi_m(\mathbf c,\mathbf d)|
    \lesssim h\bigl(|\mathbf a-\mathbf c|+|\mathbf b-\mathbf d|\bigr).
\end{equation*}
\end{lem}
The proof of Lemma~\ref{lem:Phi-structure} is given in the appendix; 
see Section~\ref{app:lem-phi-structure}.

\section{Scalar conservation laws}
\label{sec:scalar-case}

In this section, we establish the error estimate of the ESDG scheme for the scalar 
conservation laws, which reads
\begin{equation}
\label{eq:scalar-esdg-scheme}
    M^K\frac{d\vec{u}_h^K}{dt}
    +
    2\sum_{m=1}^d
    \bigl(
       S_m^K\circ F_{m,S}(\vec{u}_h^K,\vec{u}_h^K)
    \bigr)\vec{1}^K
    =
    \sum_{\gamma\in\mathcal E_K}
    ( R^{\gamma,K})^T B ^\gamma
    \left(
        \vec{f}_{\mathbf n}^{\gamma,K}
        -
        \vec{\hat{f}}_{\mathbf n}^{\gamma,K}
    \right).
\end{equation}
Before proceeding to the analysis, 
we point out a distinction between the ESDG method and the classical DG method. 
In a classical DG method, the numerical solution is naturally a piecewise polynomial. 
In contrast, the ESDG scheme \eqref{eq:scalar-esdg-scheme} 
is formulated as an evolution equation for nodal vectors, 
which are not automatically associated with functions in a prescribed polynomial space.
To employ standard polynomial approximation and inverse estimates, 
we assume that each nodal vector admits a polynomial reconstruction.

\begin{assum}
\label{ass:polynomial-space}
For each element \(K\in\mathcal T_h\), there exists a finite-dimensional polynomial space \(V_h(K)\) satisfying
\[
\mathcal P^k(K)\subset V_h(K)\subset\mathcal P^r(K),
\]
where \(r\geq k\) is fixed and independent of \(h\), such that the nodal evaluation map
\[
\mathcal N_K:V_h(K)\to\mathbb R^{N_{Q,k}},
\quad
\mathcal N_K w
:=
\bigl(w(\mathbf x_1^K),\ldots,w(\mathbf x_{N_{Q,k}}^K)\bigr)^T,
\]
is an isomorphism. Hence, every nodal vector is identified with its unique polynomial reconstruction in \(V_h(K)\).
\end{assum}

We also introduce the following a priori bound, 
which is a standard auxiliary assumption in nonlinear DG error analysis 
for nonlinear conservation laws \cite{2004-ZhangShu-SINUM}.

\begin{assum}
\label{ass:a-priori}
Let \(u_h\) denote the polynomial reconstruction provided by Assumption~\ref{ass:polynomial-space}. The numerical solution satisfies
\[
\|u-u_h\|_{L^\infty}\leq h,
\quad 0\le t\le T.
\]
\end{assum}

We first prove the error estimate under Assumptions~\ref{ass:polynomial-space} and~\ref{ass:a-priori}. 
A further discussion and justification of these assumptions will be given 
in Subsection~\ref{subsec:discussion-assumptions}.

\subsection{Error Estimate for the ESDG scheme}

\begin{thm}\label{thm:scalar-esdg}
Given a uniform convex entropy function \(U(u)\), assume that the exact solution \(u\), 
the physical fluxes \(f_m\), the entropy conservative fluxes \(f_{m,S}\), 
are sufficiently smooth, and the entropy stable fluxes \(\hat f_{\mathbf n}\) are 
locally Lipschitz continuous with respect to both arguments.
Suppose that Assumption~\ref{ass:a-priori} and~\ref{ass:polynomial-space} hold on \([0,T]\)
and that the initial error satisfies
\begin{equation*}
    \| u-u_h\|_h\lesssim h^k,\quad t=0.
\end{equation*}
Then the solution of \eqref{eq:scalar-esdg-scheme} satisfies
\begin{equation*}
    \|u-u_h\|_h\lesssim h^k,
    \quad 0\le t\le T.
\end{equation*}
\end{thm}

\begin{proof}
For each element \(K\in\mathcal T_h\), define
\[
    \vec e^K:=\vec u^K-\vec u_h^K,\quad\vec e^K_v:=\vec v^K-\vec v_h^K,
\]
where $\vec v^K:= [v(u^K_1),\cdots,v(u^K_{N_{Q,k}})]^T,\vec v_h^K:= [v(u^K_{h,1}),\cdots,v(u^K_{h,N_{Q,k}})]^T$, 
and \(v(u)=U'(u)\) is the entropy variable. 
By Taylor's expansion we have
\(
   \vec e^K_v=H^K \vec e^K,
\)
where
\[
H^K:= diag(H_1^K,\cdots,H_{N_{Q,k}}^K), 
\quad H_i^K:= \int_{0}^1 U''(u^K_{h,i}+\theta(u^K_{i}-u^K_{h,i}))d\theta.
\]
Let
\[
    f_{m,c}(a,b):=\frac{f_m(a)+f_m(b)}{2},
    \quad
    \Phi_m(a,b):=f_{m,S}(a,b)-f_{m,c}(a,b).
\]
The corresponding matrices are defined entrywise by
\[
    \bigl(F_{m,c}(\vec w^K,\vec w^K)\bigr)_{ij}
    :=f_{m,c}(w_i^K,w_j^K),
    \quad
    \bigl(\Phi_m(\vec w^K,\vec w^K)\bigr)_{ij}
    :=\Phi_m(w_i^K,w_j^K).
\]
Since \(f_{m,S}=f_{m,c}+\Phi_m\), the scheme \eqref{eq:scalar-esdg-scheme} can be rewritten as
\begin{equation}
    \label{eq:scheme-split-general}
    \begin{aligned}
    M^K\frac{d\vec u_h^K}{dt}
    &+2\sum_{m=1}^d\bigl(S_m^K\circ F_{m,c}(\vec u_h^K,\vec u_h^K)\bigr)\vec 1^K
    +2\sum_{m=1}^d\bigl(S_m^K\circ \Phi_m(\vec u_h^K,\vec u_h^K)\bigr)\vec 1^K
    \\
    &={}\sum_{\gamma\in\mathcal E_K}(R^{\gamma,K})^T B^\gamma
    \bigl(\vec f_{\mathbf n}(\vec u_h^{\gamma,K})-
    \vec{\hat f}_{\mathbf n}(\vec u_h^{\gamma,K},\vec{u}_h^{\gamma,\tilde{K}})\bigr).
    \end{aligned}
\end{equation}
Using \(S_m^K\vec 1^K=\vec 0\), we have
\[
    2\bigl(S_m^K\circ F_{m,c}(\vec w^K,\vec w^K)\bigr)\vec 1^K
    =S_m^K\vec f_m(\vec w^K), \quad \forall \vec w^K \in \mathbb{R}^{N_{Q,k}}.
\]
By Lemma~\ref{lem:local-truncation-error}, the exact solution satisfies
\begin{equation}
    \label{eq:exact-split-general}
    \begin{aligned}
    M^K\frac{d\vec u^K}{dt}
    &+\sum_{m=1}^d S_m^K\vec f_m(\vec u^K)
    +2\sum_{m=1}^d\bigl(S_m^K\circ \Phi_m(\vec u^K,\vec u^K)\bigr)\vec 1^K
    \notag\\
    &=\sum_{\gamma\in\mathcal E_K}(R^{\gamma,K})^T B^\gamma
    \bigl(\vec f_{\mathbf n}(\vec u^{\gamma,K})-
    \vec{\hat f}_{\mathbf n}(\vec u^{\gamma,K},\vec{u}^{\gamma,\tilde{K}})\bigr)
    +\vec T^K(u),
\end{aligned}
\end{equation}
where $\vec T^K(u)=M^K \vec \tau^K(u)$. Subtracting \eqref{eq:scheme-split-general} from \eqref{eq:exact-split-general}, we get
\begin{align}
        M^K\frac{d\vec e^K}{dt}
    &+\sum_{m=1}^d S_m^K\Delta\vec f_m^K
    +2\sum_{m=1}^d\bigl(S_m^K\circ\Delta\Phi_m^K\bigr)\vec 1^K \notag\\
    &=\sum_{\gamma\in\mathcal E_K}(R^{\gamma,K})^T B^\gamma
    \bigl(\Delta\vec f_{\mathbf n}^{\gamma,K}
    -\Delta\vec{\hat f}_{\mathbf n}^\gamma\bigr)
    +\vec T^K(u).
    \label{eq:error-equation-general}
\end{align}
where
\begin{align*}
    &\Delta\vec f_m^K:=\vec f_m(\vec u^K)-\vec f_m(\vec u_h^K),
    \quad 
    \Delta\Phi_m^K:=\Phi_m(\vec u^K,\vec u^K)-\Phi_m(\vec u_h^K,\vec u_h^K),
    \\
    &\Delta\vec f_{\mathbf n}^{\gamma,K}
    :=\vec f_{\mathbf n}(\vec u^{\gamma,K})-\vec f_{\mathbf n}(\vec u_h^{\gamma,K}),
    \quad
    \Delta\vec{\hat f}_{\mathbf n}^\gamma
    :=\vec{\hat f}_{\mathbf n}(\vec u^{\gamma,K},\vec{u}^{\gamma,\tilde{K}})-
    \vec{\hat f}_{\mathbf n}(\vec u_h^{\gamma,K},\vec{u}_h^{\gamma,\tilde{K}}).
\end{align*}
Left multiplying \eqref{eq:error-equation-general} by \((\vec e^K_v)^T\), we obtain
    \begin{align}
    (\vec e^K)^T H^K M^K\frac{d\vec e^K}{dt}
    =&-\sum_{m=1}^d(\vec e^K)^T H^K S_m^K\Delta\vec f_m^K
    -2\sum_{m=1}^d(\vec e^K)^T H^K\bigl(S_m^K\circ\Delta\Phi_m^K\bigr)\vec 1^K \notag
    \\
    &+\sum_{\gamma\in\mathcal E_K}(\vec e^{\gamma,K}_v)^T B^\gamma
    \bigl(\Delta\vec f_{\mathbf n}^{\gamma,K}
    -\Delta\vec{\hat f}_{\mathbf n}^\gamma\bigr)
    +(\vec e^K)^T H^K \vec T^K(u), \label{eq:energy-before-split-general}
    \end{align}
where $\vec e^{\gamma,K}_v:=R^{\gamma,K}\vec e^K_v$.
We then decompose \(S_m^K\) into its symmetric and skew-symmetric parts:
\begin{equation}
    S_m^K=\frac12\bigl(S_m^K+(S_m^K)^T\bigr)
    +\frac12\bigl(S_m^K-(S_m^K)^T\bigr).
    \label{eq:S-decomposition-general}
\end{equation}
Let $A_m^K:=\frac12\bigl(S_m^K-(S_m^K)^T\bigr)$, by \eqref{eq: nodal-SBP} and \eqref{eq:S-decomposition-general}, we have
\begin{equation}
    (\vec e^K)^T H^K M^K\frac{d\vec e^K}{dt}
    =\mathcal S_1^K+\mathcal S_2^K+\mathcal S_3^K+\mathcal S_4^K,
    \label{eq:energy-init-general}
\end{equation}
where
\begin{align*}
    \mathcal S_1^K
    &:=\sum_{\gamma\in\mathcal E_K}(\vec e^{\gamma,K}_v)^T B^\gamma
    \left(\frac12\Delta\vec f_{\mathbf n}^{\gamma,K}
    -\Delta\vec{\hat f}_{\mathbf n}^\gamma\right),\quad
    \mathcal S_2^K
    :=-\sum_{m=1}^d(\vec e^K)^T H^K A_m^K\Delta\vec f_m^K,\\
    \mathcal S_3^K
    &:=-2\sum_{m=1}^d(\vec e^K)^T H^K\bigl(S_m^K\circ\Delta\Phi_m^K\bigr)\vec 1^K,\quad
    \mathcal S_4^K
    :=(\vec e^K)^T H^K \vec T^K(u).
\end{align*}
We then define 
\[
    \| \vec{w}^K \|_{U,K}^2:= 
    (\vec{w}^K)^T H^K M^K \vec{w}^K,\quad 
    \| \vec{w} \|_{U}^2=\sum_{K\in \mathcal{T}_h}\| \vec{w}^K \|_{U,K}^2.
\]
According to the uniform convexity of the entropy \(U\),
the \(\| \cdot \|_{U,K} \) is equivalent to \(\| \cdot \|_{h,K} \).
Then \eqref{eq:energy-init-general} becomes
\begin{align*}
    \frac12\frac{d}{dt}\|\vec e^K\|_{U,K}^2
    =\mathcal S_0^K+\mathcal S_1^K+\mathcal S_2^K+\mathcal S_3^K+\mathcal S_4^K,
\end{align*}
where
\[
    \mathcal S_0^K=\frac12 (\vec e^K)^T\frac{dH^K}{dt}M^K\vec e^K.
\]
Summing over all elements gives
\begin{equation}
    \frac12\frac{d}{dt}\|\vec e\|_U^2
    =\sum_{K\in\mathcal T_h}\mathcal S_0^K
    +\sum_{K\in\mathcal T_h}\mathcal S_1^K
    +\sum_{K\in\mathcal T_h}\mathcal S_2^K
    +\sum_{K\in\mathcal T_h}\mathcal S_3^K
    +\sum_{K\in\mathcal T_h}\mathcal S_4^K.
    \label{eq:global-energy-decomposition-general}
\end{equation}
We first estimate \(\sum_K\mathcal S_0^K\). In fact,
\begin{align*}
    \sum_{K \in\mathcal T_h}\mathcal S_0^K = 
    \frac12 \sum_{K \in\mathcal T_h} \sum_{i=1}^{N_{Q,k}} \omega_i^K \partial_t H_i^K (e^K_i)^2,
\end{align*}
where
\begin{align*}
    \partial_t H_i^K &=\int_{0}^1 U'''(u^K_{h,i}+\theta(u^K_{i}-u^K_{h,i}))
    ((1-\theta)\partial_t u^K_{h,i}+\theta \partial_t u^K_{i})d \theta, \\
    |\partial_t H_i^K| &\leq \sup_s | U'''(s)| \bigl( |\partial_t u^K_i|+ |\partial_t u^K_{h,i}| \bigr),
\end{align*}
according to \eqref{eq:scalar-esdg-scheme}, using the fact that \(D_m^K\vec 1^K= \vec 0\) again, 
we have
\begin{align*}
    \partial_t u_{h,i}^K  =
    &-2\sum_{m=1}^d\sum_{j=1}^{N_{Q,k}}D_{m,ij}^K f_{m,S}(u_{h,i}^K,u_{h,j}^K)  
    \\&+\sum_{\gamma\in\mathcal E_K}\sum_{s=1}^{N_{B,k}}R_{si}^{\gamma,K}\frac{\tau_s^\gamma}{\omega_i^K}
    \bigl(f_{\mathbf n}(u_{h,s}^{\gamma,K})-\hat f_{\mathbf n}(u_{h,s}^{\gamma,K},u_{h,s}^{\gamma,\tilde{K}})\bigr) \\
    = 
    &-2\sum_{m=1}^d\sum_{j=1}^{N_{Q,k}}D_{m,ij}^K \bigl( f_{m,S}(u_{h,i}^K,u_{h,j}^K)
    - f_{m,S}(u_{h,i}^K,u_{h,i}^K)\bigr)\\
    & +\sum_{\gamma\in\mathcal E_K}\sum_{s=1}^{N_{B,k}}R_{si}^{\gamma,K}\frac{\tau_s^\gamma}{\omega_i^K}
    \bigl(f_{\mathbf n}(u_{h,s}^{\gamma,K})-\hat f_{\mathbf n}(u_{h,s}^{\gamma,K},u_{h,s}^{\gamma,\tilde{K}})\bigr).\\
\end{align*}
By Lemma~\ref{lem:scaling-sbp} and the smoothness of the numerical flux, and Assumption~\ref{ass:a-priori},
we have
\begin{align*}
    |\partial_t u^K_{h,i}| \leq & 2\sum_{m=1}^d\sum_{j=1}^{N_{Q,k}}|D_{m,ij}^K| |\bigl( f_{m,S}(u_{h,i}^K,u_{h,j}^K)  - f_{m,S}(u_{h,i}^K,u_{h,i}^K)\bigr)| \\ 
    &+\sum_{\gamma\in\mathcal E_K}\sum_{s=1}^{N_{B,k}}|R_{si}^{\gamma,K} \frac{\tau_s^\gamma}{\omega_i^K}| 
    |\bigl(f_{\mathbf n}(u_{h,s}^{\gamma,K})-\hat f_{\mathbf n}(u_{h,s}^{\gamma,K},u_{h,s}^{\gamma,\tilde{K}})\bigr)|\\
    \lesssim & h^{-1} \sum_{j=1}^{N_{Q,k}}|u_{h,i}^K-u_{h,j}^K|+ h^{-1}\sum_{s=1}^{N_{B,k}}|u_{h,s}^{\gamma,K}-u_{h,s}^{\gamma,\tilde{K}}| 
    \lesssim 1.
\end{align*}
Thus, we have the estimate for \(\mathcal{S}_0^K\),
\begin{equation}
    \sum_{K \in\mathcal T_h}\mathcal S_0^K 
    \lesssim \sum_{K \in\mathcal T_h} \sum_{i=1}^{N_{Q,k}} \omega_i^K (e^K_i)^2 
    \lesssim  \sum_{K \in\mathcal T_h} \|\vec e^K\|_{h,K}^2 
    \lesssim  \sum_{K \in\mathcal T_h} \|\vec e^K\|_{U,K}^2
    \lesssim \|\vec e\|_{U}^2.
    \label{eq: S0-Estimate-General}
\end{equation}
We then estimate \(\sum_K\mathcal S_1^K\). 
At the surface quadrature node \(\mathbf x_s^\gamma\), we denote
$e_{v,s}^{\gamma,K^{\pm}}:=v_s^\gamma-v_{h,s}^{\gamma,K^{\pm}},
    e_{s}^{\gamma,K^{\pm}}:=u_s^\gamma-u_{h,s}^{\gamma,K^{\pm}}$.
Since the outward normal of \(K^-\) on \(\gamma\) is \(\mathbf n_\gamma\), 
while the outward normal of \(K^+\) on \(\gamma\) is \(-\mathbf n_\gamma\),  we have
\begin{align*}
    \sum_{K\in\mathcal T_h}\mathcal S_1^K
    ={}&\sum_{\gamma\in\mathcal E_h}\sum_{s=1}^{N_{B,k}}\tau_s^\gamma
    \Bigg\{
    e_{v,s}^{\gamma,K^-}
    \left(
    \hat f_{\mathbf n_\gamma}(u_{h,s}^{\gamma,K^-},u_{h,s}^{\gamma,K^+})
    -\frac12\bigl(f_{\mathbf n_\gamma}(u_s^\gamma)
    +f_{\mathbf n_\gamma}(u_{h,s}^{\gamma,K^-})\bigr)
    \right)\notag\\
    &\hspace{1.0cm}
    +e_{v,s}^{\gamma,K^+}
    \left(
    \frac12\bigl(f_{\mathbf n_\gamma}(u_s^\gamma)
    +f_{\mathbf n_\gamma}(u_{h,s}^{\gamma,K^+})\bigr)
    -\hat f_{\mathbf n_\gamma}(u_{h,s}^{\gamma,K^-},u_{h,s}^{\gamma,K^+})
    \right)
    \Bigg\}.
\end{align*}
We combine the two terms associated with the entropy-stable flux and then rewrite the resulting expression by adding and subtracting the entropy-conservative flux. 
\begin{align*}
    \sum_{K\in\mathcal T_h}\mathcal S_1^K
    ={}&\sum_{\gamma\in\mathcal E_h}\sum_{s=1}^{N_{B,k}}\tau_s^\gamma
    \Bigg\{
    \bigl(v_{h,s}^{\gamma,K^+}-v_{h,s}^{\gamma,K^-}\bigr)
    \Bigl(
    \hat f_{\mathbf n_\gamma}(u_{h,s}^{\gamma,K^-},u_{h,s}^{\gamma,K^+})
    -f_{\mathbf n_\gamma,S}(u_{h,s}^{\gamma,K^-},u_{h,s}^{\gamma,K^+})
    \Bigr)\notag\\
    &
    +\bigl(v_{h,s}^{\gamma,K^+}-v_{h,s}^{\gamma,K^-}\bigr)
    f_{\mathbf n_\gamma,S}(u_{h,s}^{\gamma,K^-},u_{h,s}^{\gamma,K^+})
    -\frac12 e_{v,s}^{\gamma,K^-}
    \bigl(f_{\mathbf n_\gamma}(u_s^\gamma)
    +f_{\mathbf n_\gamma}(u_{h,s}^{\gamma,K^-})\bigr)
    \notag\\
    &\hspace{2.0cm}
    +\frac12 e_{v,s}^{\gamma,K^+}
    \bigl(f_{\mathbf n_\gamma}(u_s^\gamma)
    +f_{\mathbf n_\gamma}(u_{h,s}^{\gamma,K^+})\bigr)
    \Bigg\}.
\end{align*}
By the definitions of entropy conservative flux and entropy stable flux, we have
    \begin{align}
    \sum_{K\in\mathcal T_h}\mathcal S_1^K
    \le{}&\sum_{\gamma\in\mathcal E_h}\sum_{s=1}^{N_{B,k}}\tau_s^\gamma
    \Bigg\{
    \tilde \psi_{\mathbf n_\gamma}(u_{h,s}^{\gamma,K^+})
    -\tilde \psi_{\mathbf n_\gamma}(u_{h,s}^{\gamma,K^-})
    \notag \\
    &
    -\frac12 e_{v,s}^{\gamma,K^-}
    \bigl(f_{\mathbf n_\gamma}(u_s^\gamma)
    +f_{\mathbf n_\gamma}(u_{h,s}^{\gamma,K^-})\bigr)
    +\frac12 e_{v,s}^{\gamma,K^+}
    \bigl(f_{\mathbf n_\gamma}(u_s^\gamma)
    +f_{\mathbf n_\gamma}(u_{h,s}^{\gamma,K^+})\bigr)
    \Bigg\} \notag \\
    ={}&\sum_{\gamma\in\mathcal E_h}\sum_{s=1}^{N_{B,k}}\tau_s^\gamma
    \left(G_{\gamma,s}(e_{s}^{\gamma,K^+})
    -G_{\gamma,s}(e_{s}^{\gamma,K^-})\right), \label{eq:S1-before-G-general}
    \end{align}
where \(\tilde \psi_{\mathbf n_\gamma}(u)=\psi_{\mathbf n_\gamma}(v(u))\) and 
\begin{equation*}
    G_{\gamma,s}(\mu)
    :=\frac12 (v(u_s^\gamma)-v(u_s^\gamma-\mu))\bigl(f_{\mathbf n_\gamma}(u_s^\gamma)
    +f_{\mathbf n_\gamma}(u_s^\gamma-\mu)\bigr)
    +\tilde \psi_{\mathbf n_\gamma}(u_s^\gamma-\mu)
    -\tilde \psi_{\mathbf n_\gamma}(u_s^\gamma).
\end{equation*}
Since \(\tilde \psi_{\mathbf n_\gamma}'=v'f_{\mathbf n_\gamma}\), we have
\(
    G_{\gamma,s}(0)=G_{\gamma,s}'(0)=G_{\gamma,s}''(0)=0.
\)
By Taylor's expansion and Assumption~\ref{ass:a-priori}, we have
\begin{align*}
    \left|G_{\gamma,s}(e_{s}^{\gamma,K^-})\right|
    +\left|G_{\gamma,s}(e_{s}^{\gamma,K^+})\right| 
    \le C\left(|e_{s}^{\gamma,K^-}|^3+|e_{s}^{\gamma,K^+}|^3\right) 
    \le Ch\left((e_{s}^{\gamma,K^-})^2+(e_{s}^{\gamma,K^+})^2\right).
\end{align*}
By \eqref{eq:S1-before-G-general} and Lemma~\ref{lem:scaling-sbp}, we have
\begin{equation}
    \label{eq:S1-estimate-general}
    \begin{aligned}
        \sum_{K\in\mathcal T_h}\mathcal S_1^K
        \lesssim  h\sum_{\gamma\in\mathcal E_h}\sum_{s=1}^{N_{B,k}}\tau_s^\gamma
        \left((e_s^{\gamma,K^-})^2+(e_s^{\gamma,K^+})^2\right)
        \lesssim \|\vec e\|_h^2
        \lesssim \|\vec e\|_U^2.
    \end{aligned}
\end{equation}
We next estimate \(
\sum_{K\in\mathcal T_h}\mathcal S_2^K
\).
By Taylor's expansion, there exists \(\theta_{m,j}^K\) 
between \(u_j^K\) and \(u_{h,j}^K\) such that
% \begin{equation*}
%     f_m(u_j^K)-f_m(u_{h,j}^K)
%     =f_m'(u_j^K)e_j^K+r_{m,j}^K,
%     \quad
%     r_{m,j}^K=-\frac12 f_m''(\theta_{m,j}^K)(e_j^K)^2.
% \end{equation*}
% Equivalently,
\begin{equation*}
    \Delta\vec f_m^K
    =\vec J_m^K\circ\vec e^K+\vec r_m^K,
\end{equation*}
where
\begin{align*}
    &\vec J_m^K:=\bigl[f_m'(u_1^K),\ldots,f_m'(u_{N_{Q,k}}^K)\bigr]^T, \quad
    \vec r_m^K:=\bigl[r_{m,1}^K,\ldots,r_{m,N_{Q,k}}^K\bigr]^T.
\end{align*}
Here \(r_{m,j}^K=-\frac12 f_m''(\theta_{m,j}^K)(e_j^K)^2\). 
Define $\vec h^K:=H^K \vec{1}^K$, we have
\begin{equation}
    \begin{aligned}
        \mathcal S_2^K = & -\sum_{m=1}^d(\vec h^K \circ\vec e^K)^T A_m^K(\vec J_m^{K}\circ\vec e^K+\vec r_m^K) \\
        = &-\sum_{m=1}^d( (\vec h^K-\vec h^K_0) \circ\vec e^K)^T A_m^K(\vec J_m^K\circ\vec e^K+\vec r_m^K) \\
        &-\sum_{m=1}^d(\vec h^K_0 \circ\vec e^K)^T A_m^K((\vec J_m^K-\vec J_{m,0}^K)\circ\vec e^K+\vec r_m^K),
    \end{aligned}
    \label{eq:S2-decom}
\end{equation}
where 
\[
\vec h^K_0:=U''(u^K_0) \vec{1}^K,\quad \vec J_{m,0}^K:=f_m'(u^K_0)\vec{1}^K,
\]
and \(\mathbf x_0^K\in K\) is a fixed point and \(u^K_0:=u(\mathbf x_0^K)\). 
In the last equality of \eqref{eq:S2-decom} we have used the property of skew-symmetric matrix, i.e.,
\[
    (\vec h^K_0 \circ\vec e^K)^T A_m^K (\vec J_{m,0}^K \circ\vec e^K)
    =
    U''(u^K_0)f_m'(u^K_0) (\vec e^K)^T A_m^K \vec e^K
    =
    0.
\]
Using Assumption~\ref{ass:a-priori} again,
together with the smoothness of \(u,f_m\) and \(U\), and Lemma \ref{lem:scaling-sbp}, we have 
    \begin{align}
        \sum_{K\in\mathcal T_h}\mathcal S_2^K
        \leq & 
        \sum_{K\in\mathcal T_h}\sum_{m=1}^d
        \|(\vec h^K-\vec h^K_0) \circ\vec e^K\|_{\ell^2,K} 
        \|A_m^K\|_{\ell^2} 
        (\|\vec J_m^{K}\circ\vec e^K\|_{\ell^2,K}+\|\vec r_m^K\|_{\ell^2,K}) 
        \notag \\
        +&
        \sum_{K\in\mathcal T_h}\sum_{m=1}^d
        \|\vec h^K_0 \circ\vec e^K\|_{\ell^2,K} 
        \|A_m^K\|_{\ell^2}
        (\|(\vec J_m^K-\vec J_{m,0}^K)\circ\vec e^K\|_{\ell^2,K}+\|\vec r_m^K\|_{\ell^2,K}) 
        \notag \\
        \lesssim & 
        \sum_{K\in\mathcal T_h} h^d\|\vec{e}^K\|_{\ell^2,K}^2 
        \lesssim \sum_{K\in\mathcal T_h} \|\vec{e}^K\|_{h,K}^2
        \lesssim \|\vec e\|_h^2 \lesssim \| \vec{e}\|_U^2.  \label{eq:S2-global-estimate-general}
    \end{align}
We then estimate \(\mathcal S_3^K\).
By Assumption~\ref{ass:a-priori} and Lemma \ref{lem:Phi-structure}, we have 
\begin{equation*}
    |(\Delta\Phi_m^K)_{ij}|
    \lesssim h\bigl(|e_i^K|+|e_j^K|\bigr).
\end{equation*}
Thus, 
\begin{equation}
    \label{eq:S3-global-estimate-general}
    \begin{aligned}
        \sum_{K\in\mathcal T_h} \mathcal S_3^K
        &\lesssim \sum_{K\in\mathcal T_h}\sum_{m=1}^d\sum_{i,j=1}^{N_{Q,k}}
        |e_i^K| |H_i^K(S_m^K)_{ij}| |(\Delta\Phi_m^K)_{ij}|
        \\
        &\lesssim \sum_{K\in\mathcal T_h}h^d\sum_{j=1}^{N_{Q,k}}|e_j^K|^2
        \lesssim \sum_{K\in\mathcal T_h} \|\vec e^K\|_{h,K}^2
        \lesssim \|\vec e\|_h^2  \lesssim\|\vec e\|_U^2.
        \end{aligned}
\end{equation}
Finally, for \(\mathcal S_4^K\), by Lemma \ref{lem:scaling-sbp}, ~\ref{lem:local-truncation-error} and the Cauchy-Schwarz inequality, we have
\begin{equation}
\label{eq:S4-estimate-general}
\begin{aligned}
    \sum_{K\in\mathcal T_h}\mathcal S_4^K
    &\le \sum_{K\in\mathcal T_h}\|\vec e^K\|_{\ell^2,K}\|H^K \vec T^K(u)\|_{\ell^2,K}
    \\
    &\lesssim
    \left(\sum_{K\in\mathcal T_h} h^{-d}\|\vec e^K\|_{h,K}^2\right)^{1/2}
    \left(\sum_{K\in\mathcal T_h} h^{2k+2d} \right)^{1/2}
    \\
    &\lesssim
     h^k\|\vec e\|_h\lesssim h^k\|\vec e\|_U \lesssim h^{2k}+\|\vec e\|_U^2. 
\end{aligned}
\end{equation}
Substituting the estimates from \eqref{eq: S0-Estimate-General}, \eqref{eq:S1-estimate-general}, \eqref{eq:S2-global-estimate-general}, \eqref{eq:S3-global-estimate-general}, and \eqref{eq:S4-estimate-general} into \eqref{eq:global-energy-decomposition-general} yields
\begin{equation*}
    \frac12\frac{d}{dt}\|\vec e\|_U^2
    \lesssim \|\vec e\|_U^2+h^{2k}.
\end{equation*}
Gronwall's inequality, the initial approximation, and the equivalence of the norms complete the proof.
\end{proof}

\subsection{Discussion of the assumptions}
\label{subsec:discussion-assumptions}

We now proceed to discuss the two assumptions introduced at the outset of this section. 
First, we elaborate on the implications of the polynomial reconstruction assumption and 
establish a criterion for the existence of such a reconstruction space. 
Second, we justify the a priori \( L^\infty\) bound via a standard continuity argument.

\paragraph{Polynomial reconstruction}
Under Assumption~\ref{ass:polynomial-space}, each nodal vector is uniquely identified with a polynomial in \(V_h(K)\). This correspondence is applied in two main contexts. First, the positivity of the quadrature weights, together with the uniqueness of the polynomial reconstruction from nodal values, implies the uniform equivalence of the discrete norm to the standard $L^2(K)$ norm:
\[
    C_1\|w_h\|_{L^2(K)}^2
    \leq
    \|w_h\|_{h,K}^2
    \leq
    C_2\|w_h\|_{L^2(K)}^2,
    \quad w_h\in V_h(K),
\]
where \(C_1,C_2>0\) are independent of \(h\). Thus, the quadrature-based norm used
in our error estimate is equivalent to the usual broken \(L^2\) norm on the
reconstructed polynomial space.

Second, since \(V_h(K)\subset\mathcal P^r(K)\), where \(r\) is fixed
independently of \(h\), the standard inverse estimate yields
\[
    \|w_h\|_{L^\infty(K)}
    \leq
    C h_K^{-d/2}\|w_h\|_{L^2(K)}
    \leq
    C h_K^{-d/2}\|w_h\|_{h,K},
    \quad w_h\in V_h(K).
\]
In particular, this estimate can be applied to the reconstructed numerical
solution \(u_h\).
\begin{prop}[Polynomial reconstruction from nodal values]
\label{prop:polynomial-reconstruction}
Let \(V_k^K\) be the Vandermonde matrix associated with the volume quadrature
nodes and \(\mathcal P^k(K)\), as defined in Subsection~\ref{subsec: SBP-intro}.
Then the following statements are equivalent:
\begin{enumerate}
    \item \(V_k^K\) has full column rank.
    \item There exists an integer \(r\geq k\) and a polynomial space \(V_h(K)\)
    satisfying
    \[
        \mathcal P^k(K)\subset V_h(K)\subset\mathcal P^r(K),
        \quad
        \dim V_h(K)=N_{Q,k},
    \]
    such that the nodal evaluation map on the volume quadrature nodes is an
    isomorphism from \(V_h(K)\) onto \(\mathbb R^{N_{Q,k}}\).
\end{enumerate}
\end{prop}

The proof of Proposition~\ref{prop:polynomial-reconstruction} is given in the
appendix; see Section \ref{app:prop-reconstruction}. Next, we provide three examples illustrating the proposition.

\begin{exmp}[One-dimensional Gauss-Lobatto nodes]
    When \(d=1\), the \(k+1\) Gauss-Lobatto nodes are distinct.
    Hence the nodal evaluation map on \(\mathcal P^k(\widehat K)\) is an
    isomorphism onto \(\mathbb R^{k+1}\), and one may simply take
    \(
        V_h(\widehat{K})=\mathcal P^k(\widehat K),
    \)
    and hence Assumption~\ref{ass:polynomial-space} holds with \(r=k\).
\end{exmp}
For the next two examples, let
\[
    \widehat K
    :=
    \left\{
        (\widehat x_1,\widehat x_2):
        \widehat x_1\geq0,\quad
        \widehat x_2\geq0,\quad
        \widehat x_1+\widehat x_2\leq1
    \right\},
\]
and let \(\lambda_1,\lambda_2,\lambda_3\) be its barycentric coordinates.
We denote by \(\operatorname{Per}(a,b,c)\) the set of all distinct
permutations of \((a,b,c)\).

\begin{exmp}[Triangular quadrature nodes for \(k=1\)]
When \(k=1\), the six quadrature nodes constructed in \cite{2017-ChenShu-JCP} have barycentric
coordinates
\[
    \hat{\mathcal X}_1
    =
    \operatorname{Per}
    \left(
        0,
        a,
        b
    \right), \quad
    a=\frac12-\frac{\sqrt3}{6},
    \quad
    b=\frac12+\frac{\sqrt3}{6}.
\]
We order the nodes as
\[
\begin{aligned}
    \widehat{\mathbf x}_1=(0,a,b),\,
    \widehat{\mathbf x}_2=(0,b,a),\,
    \widehat{\mathbf x}_3=(a,0,b),\,
    \widehat{\mathbf x}_4=(b,0,a),\,
    \widehat{\mathbf x}_5=(a,b,0),\,
    \widehat{\mathbf x}_6=(b,a,0).
\end{aligned}
\]
Following the standard construction in classic finite element methods, we can construct the Lagrange 
basis as
\begin{align*}
L_1
={}&
\frac12-\lambda_1
-3\sqrt3\,\lambda_2\lambda_3(\lambda_2-\lambda_3)
-3\lambda_1\lambda_3(\lambda_3-\lambda_1)
-3\lambda_1\lambda_2(\lambda_2-\lambda_1),
\\
L_2
={}&
\frac12-\lambda_1
-3\sqrt3\,\lambda_2\lambda_3(\lambda_3-\lambda_2)
-3\lambda_1\lambda_3(\lambda_3-\lambda_1)
-3\lambda_1\lambda_2(\lambda_2-\lambda_1),
\\
L_3
={}&
\frac12-\lambda_2
-3\sqrt3\,\lambda_1\lambda_3(\lambda_1-\lambda_3)
-3\lambda_2\lambda_3(\lambda_3-\lambda_2)
-3\lambda_1\lambda_2(\lambda_1-\lambda_2),
\\
L_4
={}&
\frac12-\lambda_2
-3\sqrt3\,\lambda_1\lambda_3(\lambda_3-\lambda_1)
-3\lambda_1\lambda_2(\lambda_1-\lambda_2)
-3\lambda_2\lambda_3(\lambda_3-\lambda_2),
\\
L_5
={}&
\frac12-\lambda_3
-3\sqrt3\,\lambda_1\lambda_2(\lambda_1-\lambda_2)
-3\lambda_2\lambda_3(\lambda_2-\lambda_3)
-3\lambda_1\lambda_3(\lambda_1-\lambda_3),
\\
L_6
={}&
\frac12-\lambda_3
-3\sqrt3\,\lambda_1\lambda_2(\lambda_2-\lambda_1)
-3\lambda_1\lambda_3(\lambda_1-\lambda_3)
-3\lambda_2\lambda_3(\lambda_2-\lambda_3).
\end{align*}
Clearly \(V_h(\widehat{K}):=\operatorname{span} \left\{ L_i \right\}_{i=1}^6\) satisfies our assumption.
It is worth noting that one cannot simply take
\(V_h(\widehat{K})=\mathcal P^2(\widehat K)\), even though
\(\dim\mathcal P^2(\widehat K)=6\). Indeed, the nonzero polynomial
\[
    q
    :=
    \frac16
    -
    \bigl(
        \lambda_1\lambda_2
        +
        \lambda_2\lambda_3
        +
        \lambda_3\lambda_1
    \bigr)
\]
vanishes at all six nodes in \(\widehat{\mathcal X}_1\).
\end{exmp}

\begin{exmp}[Triangular quadrature nodes for \(k=2\)]
When \(k=2\), the ten quadrature nodes constructed in \cite{2017-ChenShu-JCP} have barycentric
coordinates
\[
    \widehat{\mathcal X}_2
    =
    \left\{
        \left(\frac13,\frac13,\frac13\right)
    \right\}
    \cup
    \operatorname{Per}
    \left(
        0,\frac12,\frac12
    \right)
    \cup
    \operatorname{Per}
    \left(
        0,
        \frac12-\frac{\sqrt{15}}{10},
        \frac12+\frac{\sqrt{15}}{10}
    \right).
\]
Since
\(
    N_{Q,2}=10
    =
    \dim\mathcal P^3(\widehat K),
\)
one may take
\[
    V_h(\widehat K)
    :=
    \mathcal P^3(\widehat K).
\]
It is easy to check
\[
    \mathcal P^2(\widehat K)
    \subset
    V_h(\widehat K)
    =
    \mathcal P^3(\widehat K),
\]
and the nodal evaluation map is an isomorphism from
\(V_h(\widehat K)\) onto \(\mathbb R^{10}\).
\end{exmp}

\begin{rem}
    For higher-order triangular quadrature nodes (e.g., $k=3,4$), as shown in \cite{2017-ChenShu-JCP}, we can still numerically construct $V_h(\hat{K})$ that satisfies Assumption \ref{ass:polynomial-space}. Due to space limitations, we omit the construction details.
\end{rem}

\paragraph{A priori assumption}
We next justify Assumption~\ref{ass:a-priori} by a standard continuity
argument. Let 
\[
    T^*
    :=
    \sup\left\{
        t\in (0,T]:
        \|u(s)-u_h(s)\|_{L^\infty}\le h,
        \quad 0\le s\le t
    \right\}.
\]
On the interval \([0,T^*]\), Assumption~\ref{ass:a-priori} holds, and hence
Theorem~\ref{thm:scalar-esdg} gives
\[
    \|u(t)-u_h(t)\|_h\lesssim h^k,
    \quad 0\le t\le T^*.
\] 
Assumption~\ref{ass:polynomial-space} allows us to use the standard 
interpolation estimate and the inverse estimate to obtain
\begin{align*}
    \|u(t)-u_h(t)\|_{L^\infty}
    &\le
    \|u(t)-\mathcal I_hu(t)\|_{L^\infty}
    +
    \|\mathcal I_hu(t)-u_h(t)\|_{L^\infty} \\
    &\lesssim
    h^{k+1}
    +
    h^{-d/2}\|\mathcal I_hu(t)-u_h(t)\|_h\\
    &\lesssim h^{k+1}+h^{k-d/2},
\end{align*}
 where \(\mathcal I_hu|_{K} \in V_h(K)\) interpolates \(u\) at \(\left\{ \mathbf{x}_i^K\right\}_{i=1}^{N_{Q,k}}\) for each $K\in \cT_h$.
If
\(
    k-\frac d2 > 1,
\)
then for sufficiently small \(h\) we have
\[
    \|u(t)-u_h(t)\|_{L^\infty}
    \le
    \frac{1}{2}h
    < h,
    \quad 0\le t\le T^*.
\]
By the continuity of \(u_h(t)\) in time, this improved estimate implies
\(T^*=T\).

\subsection{Extension to the ESOFDG scheme}
In \cite{2024-LiuLuShu-SISC}, 
the authors proposed the ESOFDG scheme to control spurious oscillations 
by adding a damping term to the original ESDG formulation.
In this subsection, we derive an error estimate for the ESOFDG scheme for the 
scalar conservation law, which reads
\begin{equation}
    \label{eq:scalar-ESOFDG}
    \begin{aligned}
        M^K\frac{d\vec u_h^K}{dt}
        &+
        2\sum_{m=1}^d
        \bigl(
        S_m^K\circ F_{m,S}(\vec u_h^K,\vec u_h^K)
        \bigr)\vec 1^K
        \\&=
        \sum_{\gamma\in\mathcal E_K}
        ( R^{\gamma,K})^T B^\gamma
        \left(
            \vec f_{\mathbf n}^{\gamma,K}
            -
            \vec{\widehat{f}}_{\mathbf n}^{\gamma,K}
        \right)    
        - \sigma^K(u_h) M^K Q^K \vec{u}_h^K.
    \end{aligned}
\end{equation}
Here, \(Q^K=I-V_0^KP_0^K\) is the projection matrix mentioned in 
Subsection~\ref{subsec: SBP-intro}.
The damping coefficient \(\sigma^K(u_h)\) is defined as follows:
\begin{equation}
    \label{eq:damp-coeff}
    \sigma^K(u_h)= \Big( \sum_{l=0}^1\frac{h_K^{2l}}{l+1} \sum_{|\alpha|=l} 
    \frac{1}{N_\Gamma} \sum_{\nu\in\partial K} [\![\partial^\alpha u_h |_\nu]\!]^2\Big)
    ^{\frac12},
\end{equation}
where the vector \(\alpha\) is the multi-index, \(N_\Gamma\) stands for the number of faces of \(\mathcal{E}_K\),
and \([\![w |_\nu]\!]\) denotes the jump of the function \(w\) on the vertex \(\nu\in K\). 
We only consider the face-adjacent neighbors of the element. For a smooth exact solution, the jumps appearing in
\eqref{eq:damp-coeff} vanish. We show that the additional damping term
does not reduce the convergence rate. To obtain the desired estimate, we first give a lemma that gives a suitable bound of the damping term. 
\begin{lem}[Bound of damping term]
Let $\sigma^K(u_h)$ be defined in \eqref{eq:damp-coeff}, we have
\label{lem:damping-term}
    \[
        \sigma^K(u_h)\lesssim  h^{k+1}+
        h^{-\frac{d}{2}} \sum_{\tilde{K}\in\Omega_K} \|\vec{e}^{\tilde{K}}\|_{h,\tilde{K}},
    \]
    where \(\Omega_K:=\{K\}\cup\{ \tilde{K}\in \mathcal{T}_h: K\text{ and }\tilde K\text{ share a face} \}\).
    Under Assumption~\ref{ass:a-priori}, this further implies
    \[
    \sigma^K(u_h)\lesssim h.
    \]
\end{lem}
The proof of Lemma \ref{lem:damping-term} is given in the appendix; see Section \ref{app:lem-damping}.
\begin{thm}
    Under the same assumption of Theorem~\ref{thm:scalar-esdg}, 
    the solution of the ESOFDG method \eqref{eq:scalar-ESOFDG} 
    satisfies
    \begin{equation*}
        \|u-u_h\|_h\lesssim h^k,
        \quad 0\le t\le T.
    \end{equation*}
\end{thm}

\begin{proof}
    Notice that
    \(
        \sigma^K(u)=0
    \)
    for smooth solution \(u\).
    So, under the same procedure as the proof of Theorem~\ref{thm:scalar-esdg} we get
    \begin{equation*}
    \frac12\frac{d}{dt}\|\vec e\|_U^2
    =\sum_{K\in\mathcal T_h}\mathcal S_0^K
    +\sum_{K\in\mathcal T_h}\mathcal S_1^K
    +\sum_{K\in\mathcal T_h}\mathcal S_2^K
    +\sum_{K\in\mathcal T_h}\mathcal S_3^K
    +\sum_{K\in\mathcal T_h}\mathcal S_4^K
    +\sum_{K\in \mathcal{T}_h} \mathcal{S}^K_{OF},
    \end{equation*}
    where
    \begin{equation}
        \label{eq:new-error-term-ESOFDG}
        \sum_{K\in \mathcal{T}_h} \mathcal{S}^K_{OF}:= \sum_{K\in \mathcal{T}_h} 
        \sigma^K(u_h) (\vec{e}^K)^T H^K M^K Q^K \vec{u}_h^K,
    \end{equation}
    and \(\mathcal S_i^K(0\leq i\leq 4)\) are the same as those in 
    \eqref{eq:global-energy-decomposition-general}. Thus,
    we only need to estimate $\mathcal{S}_{OF}^K$.
    We rewrite \eqref{eq:new-error-term-ESOFDG} as
    \begin{align*}
        \sum_{K\in \mathcal{T}_h} \mathcal{S}^K_{OF}=&\sum_{K\in \mathcal{T}_h} 
        -\sigma^K(u_h) (\vec{e}^K)^T H^K M^K Q^K \vec{e}^K
        +\sigma^K(u_h) (\vec{e}^K)^T H^K M^K Q^K \vec{u}^K \\
        =: & \sum_{K\in \mathcal{T}_h} \mathcal{O}^K_1+\sum_{K\in \mathcal{T}_h}\mathcal{O}^K_2.
    \end{align*}
    Let us first estimate \(\mathcal{O}_1^K\). By Lemmas \ref{lem:scaling-sbp}, \ref{lem:damping-term}, and the uniform convexity of the entropy, we have
    \begin{align*}
        \mathcal{O}^K_1 = & 
        -\sigma^K(u_h) (\vec{e}^K)^T H^K M^K Q^K \vec{e}^K
        \leq 
        \sigma^K(u_h) | (\vec{e}^K)^T H^K M^K Q^K \vec{e}^K | \\
        \leq &
        \sigma^K(u_h) \|\vec{e}^K\|_{\ell^2,K} \| H^K M^K Q^K \vec{e}^K\|_{\ell^2,K}
        \lesssim 
         h^{d+1}\|\vec{e}^K\|_{{\ell^2},K}^2 \lesssim \|\vec{e}^K\|_{h,K}^2.
    \end{align*}
    So we have 
    \[
         \sum_{K\in \mathcal{T}_h}\mathcal{O}^K_1 \lesssim \|\vec{e}\|_{h}^2\lesssim \|\vec{e}\|_{U}^2.
    \]
    For \(\mathcal{O}_2^K\), notice that according to the construction of \(Q^K\) and the smoothness of \(u\)
    \begin{align*}
        | (Q^K \vec{u}^K)_j |
        &= 
        |u_j - \frac{\sum_{i=1}^{N_{Q,k}} \omega_i u_i }{\sum_{i=1}^{N_{Q,k}} \omega_i}|
        = 
        |\frac{\sum_{i=1}^{N_{Q,k}} \omega_i (u_j-u_i) }{\sum_{i=1}^{N_{Q,k}} \omega_i}|
        \\
        &
        \leq 
        \frac{1}{\sum_{i=1}^{N_{Q,k}} \omega_i} (\sum_{i=1}^{N_{Q,k}} \omega_i |u_j-u_i| ) 
        \lesssim h,
    \end{align*}        
    so we have 
    \begin{align*}
        \sum_{K\in \mathcal{T}_h} \mathcal{O}_2^K 
        \leq &
        \sum_{K\in \mathcal{T}_h}
        \left| \sigma^K(u_h) (\vec{e}^K)^T H^K M^K Q^K \vec{u}^K \right| \\ 
        \lesssim &
        \sum_{K\in \mathcal{T}_h}
         \sigma^K(u_h)
        h^d\|\vec e^K\|_{\ell^2,K}
        \|Q^K\vec u^K\|_{\ell^2,K} 
        \lesssim 
        \sum_{K\in \mathcal{T}_h}
        \sigma^K(u_h)
        h^{1+\frac d2}
        \|\vec e^K\|_{h,K} \\
        \lesssim &
        \sum_{K\in \mathcal{T}_h}
        \Big(
         h^{k+2+\frac d 2}\|\vec e^K\|_{h,K}
        + h\|\vec e^K\|_{h,K}\sum_{\tilde{K}\in \Omega_K}\|\vec e^{\tilde K}\|_{h,\tilde K}
        \Big)
        \lesssim
         \, h^{2k+4}+ \|\vec e\|_{h}^2.
    \end{align*}
    It follows that
    \begin{align*}
        \sum_{K\in \mathcal{T}_h} \mathcal{S}^K_{OF}
        \lesssim \|\vec{e}\|_{U}^2
        + h^{2k+4}.
    \end{align*}
Applying the above estimate together with Gronwall’s inequality and the norm equivalence immediately yields the desired conclusion.
\end{proof}

\section{The hyperbolic system}
\label{sec:system-case}

We now extend our analysis to the system of conservation laws.
The proof follows the same consistency-stability framework as in
Section~\ref{sec:scalar-case}, with the main differences arising from
the matrix-valued entropy Hessian and flux Jacobians. We first introduce
the system counterparts of Assumptions~\ref{ass:polynomial-space}
and~\ref{ass:a-priori}.

\begin{assum}
\label{ass:system-polynomial-space}
For each element \(K\in\mathcal T_h\), let \(V_h(K)\) be the scalar
polynomial reconstruction space in
Assumption~\ref{ass:polynomial-space}, and define
\[
    \bigl[\mathcal P^k(K)\bigr]^n
    \subset
    \mathbf V_h(K):=[V_h(K)]^n
    \subset
    \bigl[\mathcal P^r(K)\bigr]^n.
\]
Then every vector of nodal values admits a unique polynomial
reconstruction in \(\mathbf V_h(K)\).
\end{assum}

\begin{assum}
\label{ass:system-a-priori}
Let \(\mathbf u_h\) denote the piecewise polynomial reconstruction
associated with Assumption~\ref{ass:system-polynomial-space}. Assume that
\[
    \|\mathbf u-\mathbf u_h\|_{L^\infty}
    \leq h,
    \quad 0\leq t\leq T.
\]
\end{assum}

\subsection{Error estimates for the ESDG scheme}
\label{subsec:system-case-error}

\begin{thm}
\label{thm:system-esdg}
Let \(U\) be a uniformly convex entropy function. Assume that the exact
solution \(\mathbf u\), the physical fluxes \(\mathbf f_m\), the entropy
conservative fluxes \(\mathbf f_{m,S}\)
are sufficiently smooth, and the entropy stable fluxes \(\hat {\mathbf f}_{\mathbf n}\) are 
locally Lipschitz continuous with respect to both arguments.
Suppose that
Assumption~\ref{ass:system-polynomial-space} and Assumption~\ref{ass:system-a-priori} 
holds, and
\begin{equation*}
    \|\mathbf u-\mathbf u_h\|_h
    \lesssim h^k, 
    \quad 
    t=0.
\end{equation*}
Then the solution of the ESDG scheme \eqref{eq:esdg-scheme} satisfies
\begin{equation*}
    \|\mathbf u-\mathbf u_h\|_h
    \lesssim h^k, 
    \quad 
    0\leq t\leq T.
\end{equation*}
\end{thm}

\begin{proof} 
For each \(K\in\mathcal T_h\), define
\[
    \vec{\mathbf e}^K
    :=
    \vec{\mathbf u}^K-\vec{\mathbf u}_h^K,
    \quad
    \vec{\mathbf e}_v^K
    :=
    \vec{\mathbf v}^K-\vec{\mathbf v}_h^K.
\]
By Taylor's expansion, we have
\(
   \vec{\mathbf e}^K_v=\mathbf H^K \vec{\mathbf e}^K,
\)
where
\begin{align*}
    &\mathbf H^K:= \operatorname{diag}(H^K_1,\cdots,H^K_{N_{Q,k}})\in \mathbb{R}^{n N_{Q,k}\times n N_{Q,k}}, \\
    &
    H_i^K:= \int_{0}^1 U''(\mathbf u^K_{h,i}+\theta(\mathbf u^K_{i}-\mathbf u^K_{h,i}))\, d\theta \in \mathbb{R}^{n\times n}.
\end{align*}
Each \(H_i^K\) is symmetric positive definite. Define the
entropy-weighted energy by
\[
    \|\vec{\mathbf e}\|_U^2
    :=\sum_{K\in\mathcal T_h}
    (\vec{\mathbf e}^K)^T\mathbf H^K\mathbf M^K\vec{\mathbf e}^K.
\]
The uniform convexity and smoothness of \(U\) imply $\|\vec{\mathbf e}\|_U \sim \|\vec{\mathbf e}\|_h$.
Under the same procedure of proof of Theorem~\ref{thm:scalar-esdg}, we can derive the error equation
\begin{equation*}
    \frac12\frac{d}{dt}\|\vec{\mathbf e}\|_U^2
    =\sum_{K\in\mathcal T_h}\mathcal S_0^K
    +\sum_{K\in\mathcal T_h}\mathcal S_1^K
    +\sum_{K\in\mathcal T_h}\mathcal S_2^K
    +\sum_{K\in\mathcal T_h}\mathcal S_3^K
    +\sum_{K\in\mathcal T_h}\mathcal S_4^K,
\end{equation*}
where
\begin{align*}
    \mathcal S_0^K
    &=\frac12 (\vec{\mathbf e}^K)^T \frac{d \mathbf H^K}{dt} \mathbf M^K \vec{\mathbf e}^K, 
    \quad
    \mathcal S_1^K
    =\sum_{\gamma\in\mathcal E_K}(\vec{\mathbf e}^{\gamma,K}_v)^T \mathbf B^\gamma
    \left(\frac12\Delta\vec{\mathbf f}_{\mathbf n}^{\gamma,K}
     -\Delta\vec{\hat{\mathbf f}}_{\mathbf n}^\gamma\right),\\ \mathcal S_2^K
     &=-\sum_{m=1}^d(\vec{\mathbf e}^K)^T \mathbf H^K \mathbf A_m^K\Delta\vec{\mathbf f}_m^K,\quad 
    \mathcal S_3^K
    =-2\sum_{m=1}^d(\vec{\mathbf e}^K)^T \mathbf H^K\bigl(\mathbf S_m^K\circ\Delta\boldsymbol\Phi_m^K\bigr)\vec{\mathbf 1}^K,\\
    \mathcal S_4^K
    &=(\vec{\mathbf e}^K)^T \mathbf H^K \vec{\mathbf T}^K(\mathbf{u}).
\end{align*}
For \(\mathcal{S}_0^K,\mathcal{S}_1^K,\mathcal{S}_3^K\) and \(\mathcal{S}_4^K\), the proof follows the same procedure as the 
scalar case, and we can obtain
\[
    \sum_{K\in\mathcal T_h}\mathcal S_0^K \lesssim \|\vec{\mathbf e}\|_U^2,
    \quad
    \sum_{K\in\mathcal T_h}\mathcal S_1^K \lesssim \|\vec{\mathbf e}\|_U^2,
    \quad
    \sum_{K\in\mathcal T_h}\mathcal S_3^K \lesssim \|\vec{\mathbf e}\|_U^2,
    \quad
    \sum_{K\in\mathcal T_h}\mathcal S_4^K \lesssim \|\vec{\mathbf e}\|_U^2+h^{2k}.
\]
We only estimate \(
\sum_{K\in\mathcal T_h}\mathcal S_2^K
\).
Similar to the scalar case, Taylor's expansion at the exact solution gives 
\begin{equation*}
    \Delta\vec{\mathbf f}_m^K
    =
    \mathbf J_m^K\vec{\mathbf e}^K
    +
    \vec{\mathbf r}_m^K,
\end{equation*}
where
\[
    \mathbf J_m^K
    :=
    \operatorname{diag}
    \bigl(
        \mathbf f_m'(\mathbf u_1^K),\ldots,
        \mathbf f_m'(\mathbf u_{N_{Q,k}}^K)
    \bigr),
\]
and the Taylor's remainder satisfies
\[
    |\mathbf r_{m,i}^K|
    \lesssim |\mathbf e_i^K|^2,
    \quad 1\leq i\leq N_{Q,k}.
\]
For a fixed point \(\mathbf{x}_0^K\in K\), we define
\begin{align*}
    H_0^K:= U''(\mathbf{u}(\mathbf{x}_0^K)),& \quad J_{m,0}^K:= \mathbf{f}_m^\prime (\mathbf{u}(\mathbf{x}_0^K)),
    \\
    \mathbf{H}_0^K := I_{N_{Q,k}} \otimes H_0^K,& \quad \mathbf{J}_{m,0}^K:= I_{N_{Q,k}}\otimes J_{m,0}^K,
\end{align*}
Then 
\begin{align*}
    \mathcal S_2^K = &-\sum_{m=1}^d(\vec{\mathbf e}^K)^T (\mathbf{H}^K - \mathbf{H}_0^K) \mathbf A_m^K 
    (\mathbf J_m^{K}\vec{\mathbf e}^K+\vec{\mathbf r}_m^K) \\
    &
     -\sum_{m=1}^d (\vec{\mathbf e}^K)^T \mathbf{H}_0^K \mathbf A_m^K 
    ((\mathbf{J}_m^{K}-\mathbf{J}_{m,0}^K)\vec{\mathbf e}^K+\vec{\mathbf r}_m^K) - \sum_{m=1}^d (\vec{\mathbf e}^K)^T \mathbf{H}_0^K \mathbf A_m^K \mathbf{J}_{m,0}^K \vec{\mathbf e}^K.
\end{align*}
By the property of the Kronecker product, we obtain
\[
\begin{aligned}
    \mathbf H_0^K\mathbf A_m^K\mathbf J_{m,0}^K
    =
    (I_{N_{Q,k}}\otimes H_0^K)
    (A_m^K\otimes I_n)
    (I_{N_{Q,k}}\otimes J_{m,0}^K)
    =
    A_m^K\otimes(H_0^KJ_{m,0}^K).
\end{aligned}
\]
Differentiating \eqref{eq:entropy-pair}, we obtain
\[
        H(\mathbf u)\mathbf f_m'(\mathbf u)=F_m''(\mathbf u)-
        \sum_{r=1}^n v_r(\mathbf u)f_{m,r}''(\mathbf u),
\]
where $f_{m,r}''$ denotes the Hessian matrix of the $r$-th component of $\mathbf{f}_m$.
Since \(F_m''(\mathbf u)\) and \(f_{m,r}''(\mathbf u), \, r=1,\ldots, n\) are
symmetric matrices, then the $H(\mathbf u)\mathbf f_m'(\mathbf u)$ is  symmetric matrix. Since \(A_m^K\) is a skew-symmetric matrix, we have
 \begin{equation*}
    \sum_{m=1}^d (\vec{\mathbf e}^K)^T \mathbf{H}_0^K \mathbf A_m^K \mathbf{J}_{m,0}^K \vec{\mathbf e}^K=0.
\end{equation*}
By a priori Assumption~\ref{ass:system-a-priori} and smoothness of the exact solution, and 
\(\|\mathbf A_m^K\|_{\ell^2}\lesssim h^{d-1}\), we have
\begin{equation*}
    \sum_{K\in\mathcal T_h}\mathcal S_2^K
    \lesssim \sum_{K\in\mathcal T_h}
    h^d\|\vec{\mathbf e}^K\|_{\ell^2,K}^2
    \lesssim \|\vec{\mathbf e}\|_h^2
    \lesssim \|\vec{\mathbf e}\|_U^2.
\end{equation*}
Finally, Gronwall's inequality and initial approximation complete the proof.
\end{proof}

\subsection{Extension to the ESOFDG scheme}

The ESOFDG scheme for conservation law systems reads
\begin{equation}
\begin{aligned}
    \mathbf M^K
    \frac{d\vec{\mathbf u}_h^K}{dt}
    &+
    2\sum_{m=1}^d
    \bigl(
        \mathbf S_m^K
        \circ
        \mathbf F_{m,S}
        (\vec{\mathbf u}_h^K,\vec{\mathbf u}_h^K)
    \bigr)
    \vec{\mathbf 1}^K\\
    &=
    \sum_{\gamma\in \mathcal E_K}
    (\mathbf R^{\gamma,K})^T
    \mathbf B^\gamma
    \left(
        \vec{\mathbf f}_{\mathbf n}^{\gamma,K}
        -
        \vec{\hat{\mathbf f}}_{\mathbf n}^{\gamma,K}
    \right)
    -
    \sigma^K(\mathbf u_h)
    \mathbf M^K\mathbf Q^K\vec{\mathbf u}_h^K.
\end{aligned}
    \label{eq:system-ESOFDG}
\end{equation}
The damping coefficient \(\sigma^K(\mathbf{u}_h)\) is defined as follows:
\begin{equation}
    \label{eq:system-damp-coeff}
    \sigma^K(\mathbf{u}_h)= \max_{1\leq s\leq n} \Big( \sum_{l=0}^1\frac{h_K^{2l}}{l+1} \sum_{|\alpha|=l} 
    \frac{1}{N_\Gamma} \sum_{\nu\in\partial K} [[(\mathbf{L}\partial^\alpha \mathbf{u}_h)_s |_\nu]]^2\Big)
    ^{\frac12}.
\end{equation}
The matrix \(\mathbf{L}\) comes from the characteristic decomposition such that
\[
    \sum_{i=1}^d n_i \mathbf{f}_i'(\overline{\mathbf{u}}_h)=\mathbf{L}^{-1} \mathbf{\Lambda} \mathbf{L},
\]
where \(\mathbf{n}=(n_1,\cdots,n_d)^T\) denotes a unit outward normal. \(\overline{\mathbf{u}}_h\) denotes 
some average of \(\mathbf{u}_h\) at the point \(\nu\in\partial K\) \cite{2024-LiuLuShu-SISC}. 
Hence, for each individual jump, the same matrix is applied to the two traces.
Moreover, a priori Assumption \ref{ass:system-a-priori} implies that \(\mathbf L\) is uniformly bounded:
\[
    \|\mathbf L\|_{\ell^2}
    \le C_L,
\]
where \(C_L\) is independent of \(h\), the element, and the face. Similar to the scalar case, the following lemma gives the estimate required for the additional
damping term. 
\begin{lem}
\label{lem:system-damping-term}
Let $\sigma^K(\mathbf u_h)$ be defined in \eqref{eq:system-damp-coeff}, we have
\[
    \sigma^K(\mathbf u_h)
    \lesssim h_K^{k+1}
    +
    h_K^{-d/2}
    \sum_{\tilde K\in\Omega_K}
    \|\vec{\mathbf e}^{\tilde K}\|_{h,\tilde K}.
\]
Under Assumption~\ref{ass:system-a-priori}, this further implies
\[
    \sigma^K(\mathbf u_h)
    \lesssim h_K.
\]
\end{lem}

\begin{proof}
    For any vector-valued piecewise polynomial \(\mathbf v_h\) and
    each pairwise jump, the assumptions on the characteristic matrix
    give
    \[
        \left|
                [[(\mathbf L\partial^\alpha\mathbf v_h)_s|_\nu]]
        \right|
        =
        \left|
            \left(
                \mathbf L
                [[\partial^\alpha\mathbf v_h|_\nu]]
            \right)_s
        \right|
        \le
        \|\mathbf L\|_{{\ell^2}}
        \left|
            [[\partial^\alpha\mathbf v_h|_\nu]]
        \right|
        \lesssim
        \left|
            [[\partial^\alpha\mathbf v_h|_\nu]]
        \right|.
    \]
    Thus, the characteristic transformation does not affect the
    jump estimates. A similar argument in
    Appendix~\ref{app:lem-damping} therefore yields the result.
\end{proof}
Now we can obtain the error estimate of the ESOFDG scheme for hyperbolic systems in the following Theorem.
\begin{thm}
\label{thm:system-ESOFDG}
For the symmetrizable system of conservation laws \eqref{eq:conservation-laws}, assume
that the solution \(\mathbf{u}\) and the flux function $\mathbf{f}_m$ are sufficiently smooth with bounded derivatives. Let $\mathbf{u}_h$ be the numerical solution of the ESOFDG scheme \eqref{eq:system-ESOFDG} and \(U\) be a uniformly convex entropy function. Assume that the entropy
conservative fluxes \(\mathbf f_{m,S}\)
are sufficiently smooth, and the entropy stable fluxes \(\hat {\mathbf f}_{\mathbf n}\) are locally Lipschitz continuous with respect to both arguments. Suppose that
Assumption~\ref{ass:system-polynomial-space} and Assumption~\ref{ass:system-a-priori} 
holds, and
\begin{equation*}
    \|\mathbf u-\mathbf u_h\|_h
    \lesssim h^k, 
    \quad 
    t=0.
\end{equation*}
Then the solution of the ESOFDG scheme \eqref{eq:system-ESOFDG} satisfies
\begin{equation*}
    \|\mathbf u-\mathbf u_h\|_h
    \lesssim h^k,
    \quad 0\leq t\leq T.
\end{equation*}
\end{thm}

\section{Numerical experiments}
\label{sec:numerical}

In this section, we show some numerical results to validate our error estimate in
Section~\ref{sec:scalar-case} and Section~\ref{sec:system-case}. We use a nonuniform mesh, which is 20\% random
perturbation of the uniform mesh for one-dimensional tests. Two-dimensional tests are performed on unstructured triangular meshes generated by 
Gmsh \cite{2009-GeuRe-GMSH}. 
The initial data are specified at the volume quadrature nodes. For the numerical experiments, we consider piecewise-polynomial approximations of degrees \(k=1,2,3\) and perform a comparison between the ESDG and ESOFDG schemes. Temporal discretization is achieved by the classical third-order SSP Runge–Kutta method \cite{2001-GottliebShu-SIREV}.
We set the time step \(\tau\) size as
\[
    \tau=\frac{\textrm{CFL}}{\sigma_0+\lambda_0} h^{\max(1,\frac{k+1}{3})},
\]
for one-dimensional problems, and the condition for two-dimensional problems is 
similar. 
Here, \(\sigma_0= \max_K \sigma^K(\mathbf{u}_h)\) defined in \eqref{eq:system-damp-coeff} and 
we set \(\sigma_0=0\) for the ESDG method,
\(\lambda_0\) is the maximum of the spectral radius of the Jacobian \(\mathbf{f}^\prime(\mathbf{u})\)
over each element.
In order to mitigate the contribution of temporal discretization errors to the total error, we adopt a fixed CFL number of $0.1$ for the numerical experiments described subsequently. For scalar conservation laws, the entropy‑stable flux is taken to be the local Lax–Friedrichs flux. In the system case, we solve the compressible Euler equations; the entropy‑conservative flux is that recommended by Chandrashekar \cite{2013-Chandra-CiCP}, and the entropy‑stable flux is realized via the local Lax–Friedrichs flux, with the extreme wave speeds suitably approximated \cite{2017-ChenShu-JCP}.

\subsection{One-dimensional problem}

\begin{exmp}
    We first consider the Burgers equation \(u_t+(\frac{u^2}{2})_x=0\) with periodic boundary condition.
    The initial condition is \(u_0(x)=e^{\cos x}\sin x+(\sin x)^2,x\in(0,2\pi).\) We adopt the entropy function \(U(u)=u^2+e^u\) and
    compute the solution at \(T=0.4\).  
\end{exmp}
The results of ESDG and ESOFDG are listed in Table~\ref{tab:burgers-esdg-esofdg}.
    It is reported that the convergence rate is below the optimal order
    and vary with the polynomial degree \(k\), which is
    better than our estimate. Such a phenomenon is broadly observed in experiments \cite{2017-ChenShu-JCP,2020-ChenShu-CSIAM,2024-LiuLuShu-SISC}.
\begin{table}[htbp]
    \centering
    \caption{Errors and orders of the ESDG and ESOFDG schemes for the one-dimensional Burgers equation
    at the final time \(T=0.4\).}
    \label{tab:burgers-esdg-esofdg}
\resizebox{\textwidth}{!}{
    \begin{tabular}{c cc cc cc|cc cc cc}
        \hline
        & \multicolumn{6}{c}{ESDG}
        & \multicolumn{6}{|c}{ESOFDG} \\
        \cline{2-13}
        & \multicolumn{2}{c}{\(k=1\)}
        & \multicolumn{2}{c}{\(k=2\)}
        & \multicolumn{2}{c}{\(k=3\)}
        & \multicolumn{2}{|c}{\(k=1\)}
        & \multicolumn{2}{c}{\(k=2\)}
        & \multicolumn{2}{c}{\(k=3\)} \\
        \cline{2-13}
        \(N\)
        & \(\|u-u_h\|_h\) & Order
        & \(\|u-u_h\|_h\) & Order
        & \(\|u-u_h\|_h\) & Order
        & \(\|u-u_h\|_h\) & Order
        & \(\|u-u_h\|_h\) & Order
        & \(\|u-u_h\|_h\) & Order \\
        \hline
        16
        & \(2.81\mathrm{E}{-01}\) & --
        & \(4.77\mathrm{E}{-02}\) & --
        & \(1.62\mathrm{E}{-02}\) & --
        & \(2.85\mathrm{E}{-01}\) & --
        & \(4.82\mathrm{E}{-02}\) & --
        & \(1.74\mathrm{E}{-02}\) & -- \\

        32
        & \(1.01\mathrm{E}{-01}\) & 1.480
        & \(1.10\mathrm{E}{-02}\) & 2.120
        & \(1.62\mathrm{E}{-03}\) & 3.322
        & \(1.01\mathrm{E}{-01}\) & 1.503
        & \(1.12\mathrm{E}{-02}\) & 2.106
        & \(1.65\mathrm{E}{-03}\) & 3.402 \\

        64
        & \(3.85\mathrm{E}{-02}\) & 1.386
        & \(2.03\mathrm{E}{-03}\) & 2.432
        & \(1.74\mathrm{E}{-04}\) & 3.213
        & \(3.86\mathrm{E}{-02}\) & 1.382
        & \(2.02\mathrm{E}{-03}\) & 2.473
        & \(1.76\mathrm{E}{-04}\) & 3.228 \\

        128
        & \(1.27\mathrm{E}{-02}\) & 1.595
        & \(4.05\mathrm{E}{-04}\) & 2.328
        & \(1.66\mathrm{E}{-05}\) & 3.390
        & \(1.27\mathrm{E}{-02}\) & 1.598
        & \(4.04\mathrm{E}{-04}\) & 2.319
        & \(1.67\mathrm{E}{-05}\) & 3.401 \\

        256
        & \(4.00\mathrm{E}{-03}\) & 1.672
        & \(6.56\mathrm{E}{-05}\) & 2.625
        & \(1.10\mathrm{E}{-06}\) & 3.923
        & \(3.99\mathrm{E}{-03}\) & 1.673
        & \(6.55\mathrm{E}{-05}\) & 2.624
        & \(1.10\mathrm{E}{-06}\) & 3.925 \\

        512
        & \(1.25\mathrm{E}{-03}\) & 1.678
        & \(1.09\mathrm{E}{-05}\) & 2.595
        & \(7.88\mathrm{E}{-08}\) & 3.798
        & \(1.25\mathrm{E}{-03}\) & 1.678
        & \(1.09\mathrm{E}{-05}\) & 2.594
        & \(7.88\mathrm{E}{-08}\) & 3.798 \\
        \hline
    \end{tabular}
}
\end{table}

\subsection{Two-dimensional problem}

\begin{exmp}
    We then test the two-dimensional Burgers equation which reads \(u_t+(\frac{u^2}{2})_x+(\frac{u^2}{2})_y=0.\)
    The initial condition is \(u_0(x,y)=0.5\sin(2\pi(x+y)),(x,y)\in(0,1)\times(0,1)\).
    We simply take the square entropy function \(U(u)=\frac{u^2}{2}\) and evolve both schemes up to
    \(T=0.1\). 
\end{exmp}
The results are tabulated in Table~\ref{tab:burgers-2d-esdg-esofdg}. For $k=1,3$, we observe that the convergence rates of both ESDG and ESOFDG are slightly above $k$ but do not reach $k+\frac 12$.
\begin{table}[htbp]
    \centering
    \caption{Errors and orders of the ESDG and ESOFDG schemes for the two-dimensional Burgers equation
    at the final time \(T=0.1\).}
    \label{tab:burgers-2d-esdg-esofdg}
\resizebox{\textwidth}{!}{
    \begin{tabular}{c cc cc cc|cc cc cc}
        \hline
        & \multicolumn{6}{c}{ESDG}
        & \multicolumn{6}{|c}{ESOFDG} \\
        \cline{2-13}
        & \multicolumn{2}{c}{\(k=1\)}
        & \multicolumn{2}{c}{\(k=2\)}
        & \multicolumn{2}{c}{\(k=3\)}
        & \multicolumn{2}{|c}{\(k=1\)}
        & \multicolumn{2}{c}{\(k=2\)}
        & \multicolumn{2}{c}{\(k=3\)} \\
        \cline{2-13}
        \(1/h\)
        & \(\|u-u_h\|_h\) & Order
        & \(\|u-u_h\|_h\) & Order
        & \(\|u-u_h\|_h\) & Order
        & \(\|u-u_h\|_h\) & Order
        & \(\|u-u_h\|_h\) & Order
        & \(\|u-u_h\|_h\) & Order \\
        \hline
        8
        & \(3.98\mathrm{E}{-02}\) & --
        & \(1.12\mathrm{E}{-02}\) & --
        & \(3.66\mathrm{E}{-03}\) & --
        & \(3.99\mathrm{E}{-02}\) & --
        & \(1.11\mathrm{E}{-02}\) & --
        & \(3.71\mathrm{E}{-03}\) & -- \\

        16
        & \(1.87\mathrm{E}{-02}\) & 1.092
        & \(3.00\mathrm{E}{-03}\) & 1.895
        & \(6.10\mathrm{E}{-04}\) & 2.585
        & \(1.87\mathrm{E}{-02}\) & 1.096
        & \(3.00\mathrm{E}{-03}\) & 1.885
        & \(6.17\mathrm{E}{-04}\) & 2.587 \\

        32
        & \(8.10\mathrm{E}{-03}\) & 1.205
        & \(6.84\mathrm{E}{-04}\) & 2.135
        & \(7.92\mathrm{E}{-05}\) & 2.946
        & \(8.11\mathrm{E}{-03}\) & 1.204
        & \(6.85\mathrm{E}{-04}\) & 2.132
        & \(7.96\mathrm{E}{-05}\) & 2.954 \\

        64
        & \(3.38\mathrm{E}{-03}\) & 1.260
        & \(1.31\mathrm{E}{-04}\) & 2.388
        & \(9.68\mathrm{E}{-06}\) & 3.033
        & \(3.38\mathrm{E}{-03}\) & 1.261
        & \(1.31\mathrm{E}{-04}\) & 2.389
        & \(9.70\mathrm{E}{-06}\) & 3.036 \\

        128
        & \(1.37\mathrm{E}{-03}\) & 1.298
        & \(2.35\mathrm{E}{-05}\) & 2.474
        & \(1.07\mathrm{E}{-06}\) & 3.174
        & \(1.38\mathrm{E}{-03}\) & 1.298
        & \(2.35\mathrm{E}{-05}\) & 2.474
        & \(1.07\mathrm{E}{-06}\) & 3.176 \\

        256
        & \(5.42\mathrm{E}{-04}\) & 1.344
        & \(3.89\mathrm{E}{-06}\) & 2.598
        & \(1.12\mathrm{E}{-07}\) & 3.256
        & \(5.42\mathrm{E}{-04}\) & 1.345
        & \(3.89\mathrm{E}{-06}\) & 2.598
        & \(1.12\mathrm{E}{-07}\) & 3.257 \\
        \hline
    \end{tabular}
}
\end{table}

\begin{exmp}
    We now consider the vortex evolution problem for two-dimensional Euler equations \cite{2017-ChenShu-JCP}.
    The computational domain is \([0,20]^2\) and the initial condition is given by
    \begin{align*}
        &w_1(\mathbf{x},0)=1-(x_2-y_2)\phi(r), \quad w_2(\mathbf{x},0)=1+(x_1-y_1)\phi(r), \\
        &\frac{p(\mathbf{x},0)}{\rho(\mathbf{x},0)}=1-\frac{\gamma-1}{2\gamma}\phi(r)^2, \quad p(\mathbf{x},0)=\rho(\mathbf{x},0)^\gamma.
    \end{align*}
    Here \(r=\sqrt{(x_1-y_1)^2+(x_2-y_2)^2}\). 
    We take \((y_1,y_2)=(10,10)\) and \(\phi(r)=\frac{5}{2\pi}e^{\frac{1-r^2}{2}}\).
    The exact solution is \(\mathbf{u}(x,y,t)=\mathbf{u}(x-t,y-t,0)\).
    We use the exact solution to prescribe boundary conditions and list the computational results at
    \(T=0.1\). 
\end{exmp}
As reported in Table~\ref{tab:2d-euler-esdg-esofdg}, the observed convergence rates remain slightly below the optimal threshold $k+1$. Furthermore, the ESOFDG scheme achieves rates that are almost indistinguishable from those of the ESDG scheme. This verifies that the inclusion of damping terms does not sacrifice accuracy when applied to the Euler systems.
\begin{table}[htbp]
    \centering
    \caption{Errors and orders of the ESDG and ESOFDG schemes for the two-dimensional Euler equation
    at the final time \(T=0.1\).}
    \label{tab:2d-euler-esdg-esofdg}
\resizebox{\textwidth}{!}{
    \begin{tabular}{c cc cc cc|cc cc cc}
        \hline
        & \multicolumn{6}{c}{ESDG}
        & \multicolumn{6}{|c}{ESOFDG} \\
        \cline{2-13}
        & \multicolumn{2}{c}{\(k=1\)}
        & \multicolumn{2}{c}{\(k=2\)}
        & \multicolumn{2}{c}{\(k=3\)}
        & \multicolumn{2}{|c}{\(k=1\)}
        & \multicolumn{2}{c}{\(k=2\)}
        & \multicolumn{2}{c}{\(k=3\)} \\
        \cline{2-13}
        \(20/h\)
        & \(\|\mathbf{u}-\mathbf{u}_h\|_h\) & Order
        & \(\|\mathbf{u}-\mathbf{u}_h\|_h\) & Order
        & \(\|\mathbf{u}-\mathbf{u}_h\|_h\) & Order
        & \(\|\mathbf{u}-\mathbf{u}_h\|_h\) & Order
        & \(\|\mathbf{u}-\mathbf{u}_h\|_h\) & Order
        & \(\|\mathbf{u}-\mathbf{u}_h\|_h\) & Order \\
        \hline
        16
        & \(3.32\mathrm{E}{-01}\) & --
        & \(8.98\mathrm{E}{-02}\) & --
        & \(2.60\mathrm{E}{-02}\) & --
        & \(3.31\mathrm{E}{-01}\) & --
        & \(9.00\mathrm{E}{-02}\) & --
        & \(2.65\mathrm{E}{-02}\) & -- \\

        32
        & \(1.54\mathrm{E}{-01}\) & 1.111
        & \(2.30\mathrm{E}{-02}\) & 1.968
        & \(2.42\mathrm{E}{-03}\) & 3.422
        & \(1.53\mathrm{E}{-01}\) & 1.110
        & \(2.30\mathrm{E}{-02}\) & 1.969
        & \(2.43\mathrm{E}{-03}\) & 3.446 \\

        64
        & \(6.45\mathrm{E}{-02}\) & 1.252
        & \(4.11\mathrm{E}{-03}\) & 2.482
        & \(1.93\mathrm{E}{-04}\) & 3.651
        & \(6.45\mathrm{E}{-02}\) & 1.249
        & \(4.11\mathrm{E}{-03}\) & 2.484
        & \(1.93\mathrm{E}{-04}\) & 3.655 \\

        128
        & \(2.24\mathrm{E}{-02}\) & 1.524
        & \(6.68\mathrm{E}{-04}\) & 2.621
        & \(1.49\mathrm{E}{-05}\) & 3.696
        & \(2.24\mathrm{E}{-02}\) & 1.523
        & \(6.68\mathrm{E}{-04}\) & 2.621
        & \(1.49\mathrm{E}{-05}\) & 3.697 \\

        256
        & \(7.66\mathrm{E}{-03}\) & 1.550
        & \(1.10\mathrm{E}{-04}\) & 2.602
        & \(1.28\mathrm{E}{-06}\) & 3.538
        & \(7.66\mathrm{E}{-03}\) & 1.550
        & \(1.10\mathrm{E}{-04}\) & 2.603
        & \(1.28\mathrm{E}{-06}\) & 3.537 \\

        512
        & \(2.46\mathrm{E}{-03}\) & 1.640
        & \(1.82\mathrm{E}{-05}\) & 2.599
        & \(1.16\mathrm{E}{-07}\) & 3.469
        & \(2.46\mathrm{E}{-03}\) & 1.640
        & \(1.82\mathrm{E}{-05}\) & 2.599
        & \(1.16\mathrm{E}{-07}\) & 3.469 \\
        \hline
    \end{tabular}
}
\end{table}

\section{Concluding Remarks}
\label{sec:conclude}
A rigorous error analysis is established for the semi-discrete entropy-stable discontinuous Galerkin method applied to scalar hyperbolic conservation laws and systems. Through a finite-difference-type consistency–stability framework, $k$-th order convergence is proved in a quadrature-based norm equivalent to the broken $L^2$  norm on the finite-dimensional reconstruction space. The analysis extends to the entropy-stable oscillation-free DG method, where the additional damping terms are shown not to degrade the convergence order. Although numerical experiments indicate observed rates up to half an order above the theoretical bound, these results furnish a rigorous convergence theory for this class of entropy-stable schemes. It remains open whether the $k$-th order bound is sharp for certain numerical examples, and whether sharper $(k+\frac 12)$-th order estimates can be derived within the variational framework of the finite element method.

\begin{appendices}
\section{Some technical proofs of Lemmas}
\label{sec:app}
In this appendix, we provide some technical proofs of lemmas in the error analysis.
\subsection{Proof of Lemma~\ref{lem:scaling-sbp}}
\label{app:lem-scaling}
We first prove the estimates for scalar nodal data. Let
\[
    F_K(\widehat{\mathbf x})
    =
    A_K\widehat{\mathbf x}+\mathbf b_K,
\]
be the affine mapping from the reference simplex \(\widehat K\) onto
\(K\). By shape regularity,
\begin{equation*}
    \|A_K\|\lesssim h_K,
    \quad
    \|A_K^{-1}\|\lesssim h_K^{-1},
    \quad
    |\det A_K|\sim h_K^d.
\end{equation*}
Then by the construction of SBP operators on the physical element \eqref{eq:affine-map-sbp},
it follows directly
\[
        \omega_j^K \sim h_K^d,\quad
        \tau_s^\gamma\sim h_K^{d-1},
\]
and
\begin{align*}
        \| D_m^K\|_{\ell^q}\lesssim h_K^{-1},
        \quad
        \| M^K\|_{\ell^q}\lesssim h_K^{d}, \quad \| R^{\gamma,K}\|_{\ell^q}\lesssim 1,
        \quad
        \| Q^K\|_{\ell^q} \lesssim 1.
\end{align*}
Consequently, for every scalar volume nodal vector
\(\vec w^{\,K}\),
\[
    \|\vec w^{\,K}\|_{h,K}^2
    =
    \sum_{j=1}^{N_{Q,k}}
    \omega_j^K|w_j^K|^2
    \sim
    h_K^d
    \sum_{j=1}^{N_{Q,k}}|w_j^K|^2
    =
    h_K^d\|\vec w^{\,K}\|_{\ell^2,K}^2.
\]
Moreover, by the orthogonal relation \eqref{eq:QK-ortho}, we have
\begin{align*}
    \|\vec w^{\,K}\|_{h,K}^2&=\|Q^K\vec w^{\,K}\|_{h,K}^2+\|(I-Q^K)\vec w^{\,K}\|_{h,K}^2
    \\
    &\geq  \|Q^K\vec w^{\,K}\|_{h,K}^2.
\end{align*}
It remains to extend the estimates to vector-valued nodal data. For
every matrix \(A\), let \(\mathbf A=A\otimes I_n\). If
\(\vec{\mathbf z}=(\mathbf z_j)_j\), then
\[
    \bigl|(\mathbf A\vec{\mathbf z})_i\bigr|
    =
    \left|
        \sum_j A_{ij}\mathbf z_j
    \right|
    \le
    \sum_j|A_{ij}|\,|\mathbf z_j|.
\]
Since the nodal dimensions and \(n\) are fixed, the scalar operator
estimates therefore imply the corresponding estimates for
\(\mathbf A\) under the block nodal \(\ell^q\)-norms. 
Finally,
\[
    \|\vec{\mathbf w}^{\,K}\|_{h,K}^2
    =
    \sum_{j=1}^{N_{Q,k}}
    \omega_j^K|\mathbf w_j^K|^2
    \sim
    h_K^d
    \|\vec{\mathbf w}^{\,K}\|_{\ell^2,K}^2,
\]
and, since \(\mathbf Q^K\) is an
\(\mathbf M^K\)-orthogonal projection,
\[
    \|\mathbf Q^K\vec{\mathbf w}^{\,K}\|_{h,K}
    \le
    \|\vec{\mathbf w}^{\,K}\|_{h,K}.
\]
This completes the proof.

\subsection{Proof of Lemma~\ref{lem:Phi-structure}}
\label{app:lem-phi-structure}
For notational simplicity, we fix \(m\) and write
\(\boldsymbol{\Phi}=\boldsymbol{\Phi}_m\). Let
$\mathbf{x}
    =
    \frac{\mathbf{a}+\mathbf{b}}{2},
    \mathbf{y}
    =
    \frac{\mathbf{a}-\mathbf{b}}{2}
$,
and define
$
    \mathbf{G}(\mathbf{x},\mathbf{y})
    :=
    \boldsymbol{\Phi}(\mathbf{x}+\mathbf{y}, \mathbf{x}-\mathbf{y})$. Since \(\mathbf{f}_{m,S}\) is consistent, we have
\[\mathbf{G}(\mathbf{x},0)=\mathbf{0},\quad D_{\mathbf{x}}\mathbf{G}(\mathbf{x},0)=\mathbf{0}.\]
And \(\boldsymbol{\Phi}\) is symmetric about $\mathbf{a}$ and $\mathbf{b}$, then 
\[D_{\mathbf{y}}\mathbf{G}(\mathbf{x},0)=\mathbf{0}.\]
By the fundamental theorem of calculus,
\begin{align*}
    D_{\mathbf{x}}\mathbf{G}(\mathbf{x},\mathbf{y})
    &=
    \int_0^1
    \frac{d}{dt}
    D_{\mathbf{x}}\mathbf{G}(\mathbf{x},t\mathbf{y})
    \,dt 
    =
    \int_0^1
    D_{\mathbf{y}}D_{\mathbf{x}}
    \mathbf{G}(\mathbf{x},t\mathbf{y})
    [\mathbf{y}]
    \,dt,
\end{align*}
where $D_{\mathbf{y}}D_{\mathbf{x}}
    \mathbf{G}(\mathbf{x},t\mathbf{y})
    [\mathbf{y}] \in \mathbb{R}^{n\times n}$, the $(i,j)$-th entry is 
    \[(D_{\mathbf{y}}D_{\mathbf{x}}
    \mathbf{G}(\mathbf{x},t\mathbf{y})
    [\mathbf{y}])_{ij}=\sum_{l=1}^n\frac{\partial^2 (\mathbf{G})_i}{\partial y_l \partial x_j}(\mathbf{x},t\mathbf{y})y_l.\]
Since
\(\boldsymbol{\Phi}\in C^3(\mathbb{R}^n\times\mathbb{R}^n)\), the
second derivatives of \(\mathbf{G}\) are uniformly bounded on the
corresponding compact set $S$. We obtain
\[\left\|
        D_{\mathbf{x}}\mathbf{G}(\mathbf{x},\mathbf{y})
    \right\|
   \lesssim |\mathbf{y}|.\]
Similarly, we also have
\[\left\|
        D_{\mathbf{y}}\mathbf{G}(\mathbf{x},\mathbf{y})
    \right\|
   \lesssim |\mathbf{y}|.\]
Hence,
\[
\left\|
        D_{\mathbf{a}}\boldsymbol{\Phi}(\mathbf{a},\mathbf{b})
    \right\|
    +
    \left\|
        D_{\mathbf{b}}\boldsymbol{\Phi}(\mathbf{a},\mathbf{b})
    \right\|
    \lesssim |\mathbf{a}-\mathbf{b}|.\]
Finally, since \(\mathbf{a},\mathbf{b},\mathbf{c},\mathbf{d}
    \in S\) and satisfy
\[
    |\mathbf{a}-\mathbf{b}|
    +
    |\mathbf{c}-\mathbf{d}|
    \lesssim h,\]
and 
\begin{align*}
    \boldsymbol{\Phi}(\mathbf{a},\mathbf{b})
    -
    \boldsymbol{\Phi}(\mathbf{c},\mathbf{d})
    =
    \int_0^1
    \Bigl(
        &D_{\mathbf{a}}\boldsymbol{\Phi}
        (\mathbf{c}+t(\mathbf{a}-\mathbf{c}),\mathbf{d}+t(\mathbf{b}-\mathbf{d}))
        [\mathbf{a}-\mathbf{c}] \\
        &+
        D_{\mathbf{b}}\boldsymbol{\Phi}
        (\mathbf{c}+t(\mathbf{a}-\mathbf{c}),\mathbf{d}+t(\mathbf{b}-\mathbf{d}))
        [\mathbf{b}-\mathbf{d}]
    \Bigr)
    \,dt,
\end{align*}
we have
\[
    \left|
        \boldsymbol{\Phi}(\mathbf{a},\mathbf{b})
        -
        \boldsymbol{\Phi}(\mathbf{c},\mathbf{d})
    \right|
    \lesssim h
    \left(
        |\mathbf{a}-\mathbf{c}|
        +
        |\mathbf{b}-\mathbf{d}|
    \right).\]
This completes the proof.

\subsection{Proof of Proposition~\ref{prop:polynomial-reconstruction}}
\label{app:prop-reconstruction}
We first prove the necessity. 
Suppose that such a space \(V_h(K)\) exists. Since
\[
    \mathcal P^k(K)\subset V_h(K),
\]
and
\[
    \mathcal N_K:V_h(K)\to\mathbb R^{N_{Q,k}},
\]
is an isomorphism, the restriction
\[
    \mathcal N_K|_{\mathcal P^k(K)}:\mathcal P^k(K)\to\mathbb R^{N_{Q,k}},
\]
is injective. Hence, the corresponding Vandermonde matrix \(V_k^K\) has full column rank.

We now prove the sufficiency. Assume that \(V_k^K\) has full column
rank. Then
\[
    \mathcal N_K:
    \mathcal P^k(K)
    \to
    U:=\mathcal N_K(\mathcal P^k(K)),
\]
is an isomorphism. Choose a complementary subspace \(W\) such that
\[
    \mathbb R^{N_{Q,k}}=U\oplus W.
\]
Since the quadrature nodes are pairwise distinct, there exists a
polynomial interpolation operator
\[
    \mathcal I_K:
    \mathbb R^{N_{Q,k}}
    \to
    \mathcal P^{N_{Q,k}-1}(K),
\]
such that
\[
    \mathcal N_K\mathcal I_K
    =
    I_{\mathbb R^{N_{Q,k}}}.
\]
Indeed, one may define
\[
    \mathcal I_K\vec a
    :=
    \sum_{i=1}^{N_{Q,k}}a_iL_i,
    \quad \forall \vec{a}=(a_1,\cdots,a_{N_{Q,k}})^T \in \mathbb{R}^{N_{Q,k}},
\]
where
\[
    L_i(\mathbf x)
    :=
    \prod_{j\ne i}
    \frac{
        (\mathbf x_i-\mathbf x_j)\cdot
        (\mathbf x-\mathbf x_j)
    }{
        |\mathbf x_i-\mathbf x_j|^2
    }.\]
Set
\[
    V_h(K)
    :=
    \mathcal P^k(K)+\mathcal I_K(W).
\]
We claim that the sum is direct. If
\(p=\mathcal I_K\vec w\) for some
\(p\in\mathcal P^k(K)\) and \(\vec w\in W\), then
\[
    \mathcal N_Kp
    =
    \mathcal N_K\mathcal I_K\vec w
    =
    \vec w
    \in U\cap W,
\]
and hence \(\vec w=0\) and \(p=0\). Moreover,
\(\mathcal N_K\) maps \(\mathcal P^k(K)\) isomorphically onto \(U\)
and \(\mathcal I_K(W)\) isomorphically onto \(W\). Therefore,
\[
    \mathcal N_K:
    V_h(K)
    \to
    U\oplus W
    =
    \mathbb R^{N_{Q,k}},
\]
is an isomorphism. In particular,
\[
    \dim V_h(K)=N_{Q,k}.
\]
Finally, by the construction above, we have
\[
    \mathcal P^k(K)
    \subset
    V_h(K)
    \subset
    \mathcal P^r(K),
    \quad
    r:=\max\{k,N_{Q,k}-1\}.
\]

\subsection{Proof of Lemma~\ref{lem:damping-term}}
\label{app:lem-damping}

Let \(\mathcal I_hu\) be the nodal interpolation function of \(u\) in the
reconstruction space \(V_h(K)\), and set
\[
    w_h:=\mathcal I_hu,
    \quad
    z_h:=u_h-w_h.
\]
One can easily check that the triangle inequality holds
\[
    \sigma^K(u_h)
    \le
    \sigma^K(z_h)+\sigma^K(w_h).
\]
Using the standard interpolation estimate we get
\begin{align*}
    h^l
    \left|
        [\![\partial^\alpha w_h|_\nu]\!]
    \right|
    =
    h^l
    \left|
        [\![\partial^\alpha(w_h-u)|_\nu]\!]
    \right| 
    \lesssim
    h^l
    \sum_{\tilde K\in\Omega_K}
    \left\|
        \partial^\alpha(w_h-u)
    \right\|_{L^\infty(\tilde K)} 
    \lesssim
    h^l h^{k+1-l}
    =
    h^{k+1}.
\end{align*}
Since the number of faces is fixed, it follows
from the definition of the damping coefficient that
\[
    \sigma^K(w_h)
    \lesssim
    h^{k+1}.
\]
For \(\sigma^K(z_h)\), the inverse estimate and the norm equivalence give
\begin{align*}
    h^l
    \left|
        [\partial^\alpha z_h|_\nu]
    \right|
    \lesssim
    h^l
    \sum_{\tilde K\in\Omega_K}
    \|\partial^\alpha z_h\|_{L^\infty(\tilde K)}
    \lesssim
    h^{-d/2}
    \sum_{\tilde K\in\Omega_K}
    \|z_h\|_{L^2(\tilde K)}
    \lesssim
    h^{-d/2}
    \sum_{\tilde K\in\Omega_K}
    \|z_h\|_{h,\tilde K}.
\end{align*}
Moreover, we have \(\vec z^{\,\tilde K}=-\vec e^{\,\tilde K}\) since \(\mathcal I_hu\) interpolates \(u\) at the volume
quadrature nodes.
Therefore,
\[
    \sigma^K(z_h)
    \lesssim
    h^{-d/2}
    \sum_{\tilde K\in\Omega_K}
    \|\vec e^{\,\tilde K}\|_{h,\tilde K}.
\]
Combining the preceding estimates yields
\[
    \sigma^K(u_h)
    \lesssim
    h^{k+1}
    +
    h^{-d/2}
    \sum_{\tilde K\in\Omega_K}
    \|\vec e^{\,\tilde K}\|_{h,\tilde K}.
\]
Moreover, under Assumption~\ref{ass:a-priori},
\[
    \|\vec e^{\,\tilde K}\|_{h,\tilde K}
    =\sqrt{\sum_{i=1}^{N_{Q,k}}\omega_i^{\tilde{K}} (e_i^{\tilde{K}})^2 }
    \lesssim
    h^{d/2}
    \|e\|_{L^\infty(\tilde K)}
    \lesssim
    h^{1+d/2}.
\]
Since the number of elements in \(\Omega_K\) is uniformly bounded,
we conclude that
\[
    \sigma^K(u_h)\lesssim h.
\]
This completes the proof.

\end{appendices}

\end{document}